    \RequirePackage{fix-cm}
    \documentclass[smallextended,envcountsect]{svjour3}       
    \smartqed  
    \usepackage{graphicx}
\usepackage[english]{babel}
\usepackage{latexsym}
\usepackage{amsfonts,amsmath,mathtools}
\def\caP{{\cal P}(N)}
\def\dv{\mathbb}
\def\llangle{\langle}
\def\rrangle{\rangle}

\def\calB{{\cal B}}
\def\calC{{\cal C}}
\def\calD{{\cal D}}
\def\0{\mbox{\sf 0}}
\def\b0{\mbox{\bf 0}}

\def\nmT{N\setminus T}
\def\nmS{N\setminus S}
\def\nmA{N\setminus A}
\def\nmM{N\setminus M}
\def\nmD{N\setminus D}
\def\smin{\setminus}  

\def\m{m} 
\def\mre{m^{*}} 
\def\mstar{m^{\star}}
\def\mg{\widetilde{m}}
\def\mgr{\widetilde{m^{*}}}
\def\mrA{m_{A}}  
\def\mrS{m_{S}}  
\def\muA{\upsilon^{\uparrow A}} 
\def\muR{\upsilon^{\uparrow R}} 
\def\mui{\upsilon^{\uparrow\{i\}}} 
\def\muo{\upsilon^{\uparrow\emptyset}} 
\def\r{r} 

\begin{document}

\title{Facets of the Cone of Totally Balanced Games\thanks{This research has been supported by the grant GA\v{C}R n.\ 16-12010S.
We are grateful to Eric Quaeghebeur for providing us with results of his
computations related to \cite{Quaeghebeur09}.}}
\author{Tom\'{a}\v{s} Kroupa \and Milan Studen\'{y}}
\institute{T. Kroupa \and M. Studen\'{y} \at
The Czech Academy of Sciences, Institute of Information Theory and Automation, Pod Vod\'arenskou v\v{e}\v{z}\'i 4, 182 08 Prague, Czech Republic\\
              \email{kroupa@utia.cas.cz}          
}

\maketitle

\begin{abstract}
The class of totally balanced games is a class of transferable-utility coalitional games providing important models of cooperative behavior used in mathematical economics. They coincide with market games of Shapley and Shubik and every totally balanced game is also representable as the minimum of a finite set of additive games. In this paper we characterize the polyhedral cone of totally balanced games by describing its facets. Our main result is that there is a correspondence between facet-defining inequalities for the cone and the class of special balanced systems of coalitions, the so-called irreducible min-balanced systems. Our method is based on refining the notion of
balancedness introduced by Shapley. We also formulate a conjecture about what are the facets of the cone of exact games, which addresses an open problem appearing in the literature.
\keywords{Coalitional game \and Totally balanced game \and Balanced system \and Polyhedral cone}
\end{abstract}

\section{Introduction}
Totally balanced games were introduced by Shapley and Shubik in their study of market games
with transferable utility \cite{ShapleyShubik69}, which arise naturally in the area of mathematical economics.
Loosely speaking, market games are coalitional games derived from a market in which the production functions are concave and continuous.
A transferable-utility coalitional game is termed totally balanced whenever each of its subgames is balanced.
The two families of coalitional games coincide, that is, a coalitional game is totally balanced if and only if
it is a market game; see \cite{ShapleyShubik69}. Kalai and Zemel \cite{KalaiZemel1982} showed that every
totally balanced game is the minimum of a finite set of additive (inessential) coalitional games.
This characterization enabled them to show that the class of totally balanced games is the same as the class
of certain games derived from graphs, the so-called flow games; cf. \cite[Theorems 1--2]{KalaiZemel1982}.
Another characterization was given by Cs\'{o}ka et al. \cite{Csoka2009}: totally balanced games are precisely risk allocation games.
Despite their relative simplicity, the class of totally balanced games thus offers a rich source of examples
of game-theoretic and economical phenomena, since it also includes minimum cost spanning tree games \cite{Bird76}, assignment games \cite{ShapleyShubik71},
linear production games \cite{Owen}, and games arising from controlled programming problems \cite{DubeyShapley}.

The notion of a balanced collection of coalitions, which was introduced by Shapley \cite{Shapley1967On-balanced-set},
is a key ingredience in dealing with totally balanced games and their subfamilies. In the literature there exist
several attempts at modifying that concept; these appear to be important in the study of special classes of games and their properties. Specifically, Cs\'{o}ka et al.\/ \cite{Csoka2011Balancedness} use exact balancedness and overbalancedness to provide new characterizations of the class of exact games. Lohmann et al.\/ \cite{Lohmann2012minimal} then employed minimal exact balanced collections and showed that only these collections
are needed to guarantee exactness of a game.

In this paper we characterize the set of totally balanced games as a convex polyhedral cone.
This can be done in several ways. Our main result, Theorem \ref{thm.total}, offers a complete characterization of facet-defining inequalities of the cone; each of these is associated with a special min-balanced system.
The paper is structured as follows. We introduce our notation and terminology in Section \ref{sec.basic-notions},
where we also review basic facts about minimal balanced systems of coalitions in the style of Shapley \cite{Shapley1967On-balanced-set}. In Section \ref{sec.our-conventions} we explain our framework for analyzing cones of set functions and introduce our machinery for manipulating valid linear inequalities. This approach makes it possible to show that important cones of set functions are closed under the operation of reflection (Definition~\ref{def.reflection}), which result even cannot be formulated with the usual definition of a coalitional game as a set function vanishing at the empty set. Moreover, facet-defining inequalities are often coupled into pairs of certain conjugate inequalities (Lemma~\ref{lem.facet-conjugate}). For the purposes of this paper it is useful to work with linear inequalities in a special form, which is derived from min-balanced systems and normalized in a particular way; this procedure is described in Section~\ref{ssec.min-balan-ineq}. The crucial notion, which is necessary for the formulation of our main theorem, is the concept of an \emph{irreducible min-balanced system} discussed in Section~\ref{sec.irreducible}. Every other min-balanced system (= a reducible one) is redundant for the description of totally balanced cone (Corollary~\ref{cor.irreducible}). Section \ref{sec.main-result} contains the main result (Theorem~\ref{thm.total}) saying that facets of the totally balanced cone correspond to irreducible min-balanced set systems, together with a series of lemmas leading to its proof. Interestingly enough, one of the by-products of our research is a contribution to the question studied by Lohmann et al.\/ in \cite{Lohmann2012minimal}, namely what are the facet-defining inequalities for the cone of exact games. We formulate a conjecture about the form of those facet-defining inequalities in Section \ref{sec.conjecture}. Appendix contains a list of min-balanced systems of sets associated with facet-defining inequalities for small cardinalities of the player set.

\section{Basic notions and results}\label{sec.basic-notions}
Throughout the paper we assume that the reader is familiar with basic concepts and facts from polyhedral geometry; see \cite{Schrijver98,BK92}, for example.

We are going to use standard notions and results from cooperative game theory; see \cite{PelegSudholter07}. For simplicity, we assume that a finite \emph{player set} $N$ contains $n\geq 2$ players which are denoted by the first $n$ letters of English alphabet.
Thus, for example, we write $N=\{a,b,c,d\}$ in case of $4$ players. Subsets of $N$ are called \emph{coalitions}.

Any $n$-dimensional real vector $[x_{i}]_{i\in N}$ is called a {\em payoff allocation} and by ${\dv R}^{N}$ we denote the collection of all such vectors. Given $A\subseteq N$,
the symbol $\chi_{A}\in {\dv R}^{N}$ is the incidence vector of $A$ defined by
\[
(\chi_{A})_{i} ~\coloneqq~
\begin{cases}
\,1 & \mbox{for $i\in A$}, \\
\,0 & \mbox{for $i\in\nmA$}.
\end{cases}
\]
The zero vector in ${\dv R}^{N}$ will be denoted by \0.

Given a non-empty coalition $A\subseteq N$, the symbol $\mathcal{P}(A)$ will
denote the collection 
of all its subsets.
The symbol ${\dv R}^{{\cal P}(A)}$ will be used to denote the collection of
real-valued set functions on subsets of $A$, that is, mappings $\m\colon \mathcal{P}(A)\to {\dv R}$.
Given $\m\in {\dv R}^{{\cal P}(N)}$ and $\emptyset\neq A\subseteq N$, the
restriction of $\m$ to $\mathcal{P}(A)$ will be denoted by $\mrA$; formally,
$\mrA\in {\dv R}^{{\cal P}(A)}$ is defined by $\mrA(S)\coloneqq \m(S)$ for any $S\subseteq A$.

The zero vector in ${\dv R}^{\caP}$ will be denoted by \b0 and the scalar product of two elements
$\theta$ and $\m$ in the linear space ${\dv R}^{\caP}$ by
$\llangle\theta,\m\rrangle\coloneqq \sum_{S\subseteq N} \theta(S)\cdot \m(S)$\,.
Given $A\subseteq N$, the symbol $\delta_{A}$ will denote its set indicator in ${\dv R}^{\caP}$
and $\muA\in {\dv R}^{\caP}$ the indicator of its
supersets:
$$
\delta_{A}(S)\coloneqq
\left\{
\begin{array}{cl}
1 & \mbox{if $S=A$},\\
0 & \mbox{for other $S\subseteq N$},
\end{array}
\right.
\qquad
\muA(S)\coloneqq
\left\{
\begin{array}{cl}
1 & \mbox{if $A\subseteq S$},\\
0 & \mbox{for other $S\subseteq N$}.
\end{array}
\right.
$$
The set functions $\muA$ for $A\subseteq N$ with $|A|\leq 1$ then form a basis
of the linear subspace $L(N)\subseteq {\dv R}^{\caP}$
of {\em modular set functions\/} $m$  satisfying
$$
 \m(C\cup D)+\m(C\cap D)=\m(C)+\m(D), \quad \mbox{for all $C,D\subseteq N$}.
$$
The linear space $L(N)$ has the dimension $1+|N|$.

Given a set ${\cal S}\subseteq {\dv R}^{\caP}$, its {\em dual cone\/}
is $${\cal S}^{*}\coloneqq\{\, \m\in {\dv R}^{\caP}\,:\ \llangle\theta,\m\rrangle\geq 0\quad
\mbox{for any $\theta\in {\cal S}$}\,\}.$$ A well-known elementary fact  is that
${\cal C}\subseteq {\dv R}^{\caP}$ is a non-empty closed convex cone iff ${\cal C}={\cal C}^{**}$, which
happens iff ${\cal C}={\cal S}^{*}$ for some ${\cal S}\subseteq {\dv R}^{\caP}$; see
for example \cite[Consequence~1]{Studeny93}. Thus, if one shows, for a non-empty polyhedral cone
${\cal C}\subseteq {\dv R}^{\caP}$ and ${\cal D}\subseteq {\dv R}^{\caP}$ that ${\cal D}={\cal C}^{*}$ then this implies that ${\cal C}={\cal D}^{*}$ and, moreover, that ${\cal C}$ and ${\cal D}$ are mutually dual polyhedral cones.
Another well-known fact in polyhedral geometry
is that the face-lattices of dual cones are anti-isomorphic; see \cite[Theorem\,7.41]{BK92}.
In particular, if ${\cal C}$ is a pointed cone, which means $-{\cal C}\cap{\cal C}=\{\b0\}$,
then the facets of ${\cal D}$ are in bijection with the extreme rays of ${\cal C}$.

A ({\em transferable-utility coalitional\,}) {\em game} over (a set of players) $N$ is modeled by a real function $\m\colon \mathcal{P}(N)\to {\dv R}$ such that $\m(\emptyset)=0$.
We will use ${\cal G}(N)$ to denote the collection of all such functions; any $\m\in{\cal G}(N)$
will briefly be called a ``game". If $\emptyset\neq A\subseteq N$ then the restriction $\mrA$ to $\mathcal{P}(A)$ is called
a {\em subgame} of the game $m\in {\cal G}(N)$.

The {\em core $C(\m)$ of\/} a game $\m\in{\cal G}(N)$ is the set of all Pareto efficient and coalitionally rational
payoff allocations, that is, formally
$$
C(\m) \coloneqq \{\,[x_{i}]_{i\in N}\in {\dv R}^{N} \,:\ \sum_{i\in N} x_{i}=\m(N) ~\&~
\sum_{i\in S}x_{i} \geq \m(S)\; \text{for all $S\subseteq N$}\}\,.
$$
We say that a game $\m\in{\cal G}(N)$ is
\begin{itemize}
\item {\em balanced} if $C(\m)\neq\emptyset$;
\item {\em totally balanced} if every subgame of $\m$ is balanced;
\item {\em exact} if, for each coalition $S\subseteq N$, there exists a payoff allocation
$[x_{i}]_{i\in N} \in C(\m)$ in the core that is tight for $S$, which means that $\sum_{i\in S}x_{i} = \m(S)$.
\end{itemize}
We introduce the following notation:
\begin{align*}
\mathcal{B}(S) &\coloneqq \{ \m\in {\cal G}(S) :\text{$m$ is balanced}\},\quad \text{for any $\emptyset\neq S\subseteq N$},\\
\mathcal{T}(N) &\coloneqq \{ \m\in {\cal G}(N): \text{$m$ is totally balanced}\}, \\
\mathcal{E}(N) &\coloneqq \{ \m\in {\cal G}(N) : \text{$m$ is exact}\}.
\end{align*}
\noindent
Recall from \cite[Theorem\,1]{KalaiZemel1982} that $\m\in {\cal T}(N)$ iff it has a finite min-representation, which
means there exists a nonempty finite ${\cal X}\subseteq {\dv R}^{N}$ such that
$$
\m(S)~=~\min_{x\in {\cal X}}\, \sum_{i\in S} x_{i}\qquad \mbox{for any $S\subseteq N$}.
$$
It is well-known that $\m\in {\cal E}(N)$ iff it has a min-representation
$\emptyset\neq {\cal X}\subseteq C(\m)$; see \cite[Proposition\,1]{SK2018}, for example. Hence, we get the inclusions $\mathcal{E}(N)\subseteq \mathcal{T}(N)\subseteq \mathcal{B}(N)\subseteq {\dv R}^{{\cal P}(N)}$.

Moreover, all these sets are 
polyhedral cones in the linear space ${\dv R}^{{\cal P}(N)}$. Specifically, the fact that ${\cal B}(N)$ is determined by finitely many linear inequalities is a consequence of classic results by Bondareva  \cite{Bondareva63} and Shapley \cite{Shapley1967On-balanced-set}, recalled in later Lemma~\ref{lem.balanced-cone}. The cone ${\cal T}(N)$ is polyhedral since it follows immediately from the definition that
$$
{\cal T}(N) = \bigcap_{\emptyset\neq S\subseteq N}\{ \m\in {\cal G}(N)\,:\ \mrS\in {\cal B}(S)\, \}.
$$
Finally, ${\cal E}(N)$ is a polyhedral cone by the results contained in \cite{Csoka2011Balancedness} or \cite{Lohmann2012minimal}. Since all the discussed cones are polyhedral, each of them is fully determined by finitely many linear inequalities only. It turns out that the facet-defining inequalities correspond to special set systems, which we define in the next section.

\subsection{Min-balanced systems}
Any subset ${\cal B}$ of ${\cal P}(N)$ is called a {\em set system}.
The union of sets in $\calB$ will be denoted by $\bigcup\calB$ and their intersection by $\bigcap\calB$. We will call $\bigcup \calB$ the \emph{carrier} of $\calB$. Minimality of set systems is always understood in the sense of their inclusion as subsets of ${\cal P}(N)$.

\begin{definition}\label{def.min-balanced}\rm
Let $\calB\subseteq\caP$ be a non-empty set system with a carrier $M\subseteq N$. We say that $\calB$ is \emph{min-balanced} if it is a minimal set system in $\caP$ satisfying the condition:
\begin{equation}
\text{The vector $\chi_{M}$ belongs to the conic hull of
 $\{ \chi_{S}\in {\dv R}^{N}:\ S\in\calB\}$.} \tag{$\dagger$}
\end{equation}
If $\calB\subseteq\caP$ is a min-balanced system whose carrier is $M$, then we briefly say that $\calB$ is \emph{min-balanced on M}.
A min-balanced system ${\cal B}$
with $|\calB|\geq 2$ will be named {\em non-trivial}.
\end{definition}

Clearly, a permutation of players transforms a non-trivial min-balanced system also to a non-trivial
min-balanced system. A catalogue of permutational types of non-trivial min-balanced system on $N$, where $2\leq |N|\leq 4$, can be found in Appendix.

\begin{lemma}\em\label{lem.equiv-min-bal}
A non-empty set system ${\cal B}\subseteq\caP$ is min-balanced
if and only if the following two conditions hold:
\begin{itemize}
\item[(i)] There exist strictly positive coefficients $\lambda_{S}>0$ for $S\in\calB$ such that
$$
\chi_{\bigcup\calB}=\sum_{S\in\calB} \lambda_{S}\cdot\chi_{S}.
$$
\item[(ii)] The incidence vectors $\{ \chi_{S}\in {\dv R}^{N}:\ S\in\calB\}$ are linearly independent.
\end{itemize}
Hence, a non-empty ${\cal B}\subseteq\caP$ is min-balanced iff it is a minimal set system
satisfying (i).
\end{lemma}

The condition (i) is the {\em balancedness\/} condition from \cite{Shapley1967On-balanced-set}. Therefore, we also say that $\calB$ is {\em balanced on $M$} whenever (i) holds with  $M=\bigcup\calB$. Similar terminology is often used in game-theoretical literature; see, for example, an equivalent concept of $S$-balanced collection \cite[Section\,2-3]{SS2016nucleolus}.
The last claim in Lemma~\ref{lem.equiv-min-bal} motivated our terminology, which follows the usual game-theoretical terminology.
The condition (ii) is equivalent to the minimality and implies
the uniqueness of the coefficients $\lambda_{S}$ in (i).

\begin{proof}
Let $M$ denote the carrier of $\calB$.
To show the necessity of (i) write $\chi_{\bigcup\calB}=\chi_{M}=\sum_{S\in\calB} \lambda_{S}\cdot\chi_{S}$ with
$\lambda_{S}\geq 0$ by \mbox{($\dagger$)}. If $\lambda_{S}$ vanishes for some $S$, then take $\calB^{\prime}=\{T\in\calB\,:\,\lambda_{T}>0\}$
and get a contradictory conclusion that $\calB^{\prime}$ is a strict subsystem of $\calB$ satisfying \mbox{($\dagger$)}.
The necessity of (ii) can then be shown by contradiction: otherwise a non-vanishing system of real coefficients $\{\gamma_{S}\,:\,S\in\calB\}$ exists such that $\sum_{S\in\calB} \gamma_{S}\cdot\chi_{S}=\0\in {\dv R}^{N}$. For any $\varepsilon\geq 0$ put $\lambda^{\varepsilon}_{S}\coloneqq\lambda_{S}+\varepsilon\cdot\gamma_{S}$
and observe that $\chi_{M}=\sum_{S\in\calB} \lambda^{\varepsilon}_{S}\cdot\chi_{S}$.
Since all $\lambda_{S}$'s are strictly positive by (i), there exists a maximal $\varepsilon>0$  such that all $\lambda^{\varepsilon}_{S}$ are non-negative. Put $\calB^{\prime}=\{T\in\calB\,:\,\lambda^{\varepsilon}_{T}>0\}$ and derive the contradiction
analogously.

Conversely, if both (i) and (ii) hold, then
$\chi_{M}=\chi_{\bigcup\calB}=\sum_{S\in\calB} \lambda_{S}\cdot\chi_{S}$ with $\lambda_{S}>0$.
Assume for a contradiction that ${\cal C}\subset \calB$ exists such that $\chi_{M}=\sum_{S\in{\cal C}} \nu_{S}\cdot\chi_{S}$ with $\nu_{S}\geq 0$, $S\in{\cal C}$. Put $\nu_{S}=0$ for $S\in \calB\setminus{\cal C}$ and
note $\0=\sum_{S\in\calB} (\lambda_{S}-\nu_{S})\cdot\chi_{S}$, which contradicts (ii).

The last claim in Lemma~\ref{lem.equiv-min-bal} 
is a direct consequence of the proven equivalence.
\qed \end{proof}

Note that the balancedness condition (i) in Lemma~\ref{lem.equiv-min-bal}
cannot be weakened to \mbox{($\dagger$)}. Indeed, ($\dagger$) with $M=\bigcup\calB$ and (ii)
do not imply (i) as the following example shows.
Put $N=\{a,b,c\}$ and $\calB=\{\, \{a\}, \{b\},\{b,c\}\,\}$. Then
$\chi_{\bigcup\calB}=1\cdot\chi_{\{a\}}+0\cdot\chi_{\{b\}}+1\cdot\chi_{\{b,c\}}$, but (i) is not true.
We now collect  basic facts about non-trivial min-balanced systems.

\begin{lemma}\label{lem.min-balanced}\rm
Let $\calB\subseteq\caP$ be a non-trivial min-balanced system. Then the following conditions hold:
\begin{itemize}
\item $\emptyset,\bigcup\calB\not\in\calB$ and $|\bigcup\calB|\geq 2$,
\item $\bigcap\calB=\emptyset$,
\item there are at most $|\bigcup\calB|$ sets in $\calB$.
\end{itemize}
\end{lemma}

\begin{proof}
Since $\calB$ is non-trivial, it is necessarily non-empty, so Lemma~\ref{lem.equiv-min-bal} applies. Then Lemma~\ref{lem.equiv-min-bal}(ii) implies that $\emptyset\notin\calB$. As $\calB$ contains at least two non-empty sets, necessarily $|\bigcup\calB|\geq 2$. Assume for contradiction that $\bigcup\calB\in\calB$. Then the non-triviality assumption $|\calB|\geq 2$ and (i) imply together that $\chi_{\bigcup\calB}$ can be expressed in two different ways as a linear combination of $\{ \chi_{S}\in {\dv R}^{N}:\ S\in\calB\}$,
which contradicts (ii). Hence, $\bigcup\calB\not\in\calB$.

By non-triviality we have $\bigcup\calB\setminus \bigcap\calB\neq\emptyset$. Assume for contradiction that $\bigcap\calB\neq \emptyset$. Consider $i\in\bigcap\calB$, $j\in \bigcup\calB\setminus \bigcap\calB$ and choose $T\in\calB$ with $j\not\in T$. Then $i\in T$ and one can write using the formula in Lemma~\ref{lem.equiv-min-bal}(i),
$$
\sum_{S\in\calB} \lambda_{S}=\sum_{S\in\calB:i\in S} \lambda_{S} =
(\chi_{\bigcup\calB})_{i}= 1=(\chi_{\bigcup\calB})_{j}=\sum_{S\in\calB:j\in S} \lambda_{S}.
$$
This implies $\sum_{S\in\calB:j\not\in S} \lambda_{S}=0$, which gives a contradictory conclusion that simultaneously $\lambda_{T}>0$ and $\lambda_{T}=0$. Hence, necessarily $\bigcap\calB=\emptyset$.

Consider $k=|\bigcup\calB|$. Since the linear space ${\dv R}^k$ is $k$-dimensional, there are at most $k$ linearly independent vectors in it. Thus, by Lemma~\ref{lem.equiv-min-bal}(ii), $\calB$ can have at most $k$ sets.
\qed \end{proof}

Every set-theoretic partition of a non-empty subset of $N$ into non-empty blocks is a min-balanced system. A  more general example, which is not a partition, is as follows.

\begin{example} \label{exa.minbal}
Let $N\coloneqq \{a,b,c,d,e\}$. Put $\calB\coloneqq \{\{a,b\},\{a,c\},\{a,d\},\{b,c,d\}\}$
and define
$$
\lambda_S \coloneqq
\begin{cases}
\,\frac{1}{3} & \mbox{for $S\in \{\{a,b\},\{a,c\},\{a,d\}\}$}, \\
\,\frac{2}{3} & \mbox{for $S=\{b,c,d\}$}.
\end{cases}
$$
One can check using Lemma~\ref{lem.equiv-min-bal} that $\calB$ is min-balanced and its carrier is $\{a,b,c,d\}$.
\end{example}

\section{Our framework for dealing with linear inequalities}\label{sec.our-conventions}
The cones of games introduced in Section \ref{sec.basic-notions} are not full-dimensional in
${\dv R}^{\caP}$. This implies that their facet-defining inequalities have several equivalent writings in this linear space. We are going to introduce useful conventions in order to establish a (one-to-one) correspondence between
(facet-defining) inequalities and certain set systems.

\subsection{Taking the empty set into consideration}\label{ssec.zero-allow}
Traditional game-theoretical literature deals with set functions
$\m:\caP\to {\dv R}$ satisfying $\m(\emptyset)=0$, which typically leads to
restricting considerations to the linear space ${\dv R}^{\caP\setminus\{\emptyset\}}$.
This restriction, however, causes later formal complications and hides some important symmetries. 
To reveal those symmetries we intentionally consider the entire space ${\dv R}^{\caP}$ and extend our cones of games
so that the resulting cones contain constant set functions.
\smallskip

Specifically, for every $\m\in {\dv R}^{\caP}$, a shifted set function $\mg\in {\dv R}^{\caP}$ defined by
\begin{equation}
\mg (S) \coloneqq \m(S)-\m(\emptyset)\qquad \mbox{for every $S\subseteq N$}
\label{eq.zero-projection}
\end{equation}
is a game over $N$ and we define:
\begin{align*}
B(N) &\coloneqq \{\, \m\in {\dv R}^{\caP}\ :\ \mg\in {\cal B}(N)\, \},\\
T(N) &\coloneqq \{\, \m\in {\dv R}^{\caP}\ :\ \mg\in {\cal T}(N)\, \},\\
E(N) &\coloneqq \{\, \m\in {\dv R}^{\caP}\ :\ \mg\in {\cal E}(N)\, \}.
\end{align*}
All these cones are full-dimensional in ${\dv R}^{\caP}$.
A few basic observations about these cones are below.
Recall that by a {\em tight valid\/} inequality for a cone is meant such a valid inequality for
its vectors which holds with equality for at least one vector in the cone.
Note that all the considered cones have the dimension at least 1. Therefore, their facets are non-empty,
which implies that every facet-defining inequality for them is a tight valid inequality.

\begin{lemma}\label{lem.zero-allow}\rm
The space $L(N)$ of modular set functions is the (shared) linearity space for cones $B(N)$, $T(N)$ and $E(N)$. Analogously, the shared linearity space for cones ${\cal B}(N)$, ${\cal T}(N)$ and ${\cal E}(N)$ is the space
of modular games \[{\cal L}(N)\coloneqq L(N)\cap{\cal G}(N).\]
Tight valid inequalities for each of these six cones 
have the following form in  ${\dv R}^{\caP}$:
\begin{equation}
\underbrace{\sum_{S\subseteq N} \alpha (S)\cdot \m(S)}_{=\llangle \alpha,\m\rrangle}\geq 0\qquad
\mbox{required for $\m\in {\dv R}^{\caP}$}\,,
\label{eq.inequal}
\end{equation}
where the coefficient vector $\alpha\in {\dv R}^{\caP}$ satisfies
$\sum_{S\subseteq N:\, i\in S} \alpha (S)=0$ for any $i\in N$. Moreover, the coefficient vectors of tight valid inequalities for cones $B(N)$, $T(N)$ and $E(N)$ even
satisfy
\begin{equation}
\sum_{S\subseteq N} \alpha (S)=\llangle \alpha,\muo\rrangle=0,\quad
\sum_{S\subseteq N:\, i\in S} \alpha (S)=\llangle \alpha,\mui\rrangle=0~~
\mbox{for any $i\in N$.}
\label{eq.o-stan}
\end{equation}
\end{lemma}

\begin{proof}
The first observation is that, for every $i\in N$, one has $\mui,-\mui\in {\cal E}(N)$,
with singleton cores $\{\chi_{\{i\}}\}$ and $\{-\chi_{\{i\}}\}$. Thus, $\mui\in -{\cal E}(N)\cap{\cal E}(N)$ for any $i\in N$, which implies ${\cal L}(N)\subseteq -{\cal E}(N)\cap{\cal E}(N)$. Since constant set functions
$\pm\muo\in {\dv R}^{\caP}$ belong to $E(N)$ one analogously observes $L(N)\subseteq -E(N)\cap E(N)$.

The second observation is that the only $\r\in -B(N)\cap B(N)$ satisfying $\r(S)=0$ for $S\subseteq N$, $|S|\leq 1$,
is the zero game \b0. To this end realize that every balanced game $\mg\in {\cal B}(N)$ satisfies
$\mg(N)\geq \mg(T)+\sum_{i\in \nmT}\mg(\{i\})$ for any $\emptyset\neq T\subset N$. Indeed,
having $x=[x_{i}]_{i\in N}$ in the core $C(\mg)$, one can write
\begin{align*}
\mg(N) & =\sum_{i\in N} x_{i} = \sum_{i\in T} x_{i}+  \sum_{i\in \nmT} x_{i} ~\geq~
\mg(T)+  \sum_{i\in \nmT} x_{i} \\ & \geq~ \mg(T)+  \sum_{i\in \nmT} \mg(\{i\}).
\end{align*}
Thus, for any $\r$ in $-B(N)\cap B(N)$ with $\r(S)=0$ for $S\subseteq N$, $|S|\leq 1$,
both $\r$ and $-\r$ can be taken in place of $\mg$, which results in
$\r(N)=\r(T)+\sum_{i\in \nmT} \r(\{i\})=\r(T)$. Taking in place of $T$ a singleton
gives $\r(N)=0$; then taking general $\emptyset\neq T\subset N$ gives the rest.

This allows one to show $-B(N)\cap B(N)\subseteq L(N)$. Indeed, given
$\m\in -B(N)\cap B(N)$, put $\r\coloneqq  \m-\m(\emptyset)\cdot\muo +\sum_{i\in N}
[\m(\emptyset)-\m(\{i\})]\cdot \mui$, realize that $\r$ satisfies the
above conditions and derive $\r=\b0$ to observe that $\m\in L(N)$. An analogous consideration
with games leads to $-{\cal B}(N)\cap {\cal B}(N)\subseteq {\cal L}(N)$.

One has ${\cal E}(N)\subseteq {\cal T}(N)\subseteq {\cal B}(N)$, the same holds for
their multiples by $(-1)$, giving
$$
{\cal L}(N) \subseteq -{\cal E}(N)\cap {\cal E}(N)\subseteq -{\cal T}(N)\cap {\cal T}(N)\subseteq
-{\cal B}(N)\cap {\cal B}(N)\subseteq {\cal L}(N)\,;
$$
analogously $E(N)\subseteq T(N)\subseteq B(N)$ shows that $L(N)$ is the shared linear space
for the extended cones $E(N)$, $T(N)$ and $B(N)$.

Let ${\cal K}$ be any of the above discussed cones in ${\dv R}^{\caP}$.
Any linear inequality for $\m\in {\dv R}^{\caP}$ can be re-written in the form
$\llangle \alpha,\m\rrangle\geq k$ with $\alpha\in {\dv R}^{\caP}$ and $k\in {\dv R}$; otherwise we multiply
the inequality by $(-1)$. Assuming it is a tight valid inequality for ${\cal K}$ there exists $\m_{0}\in {\cal K}$
satisfying $\llangle \alpha,\m_{0}\rrangle =k$.
Because ${\cal K}$ is a cone, for every $\varepsilon>0$, one has
$\varepsilon\cdot\m_{0}\in {\cal K}$ and
$$
\varepsilon\cdot k=\varepsilon\cdot \llangle \alpha,\m_{0}\rrangle =\llangle \alpha,\varepsilon\cdot\m_{0}\rrangle\geq k
\quad\Rightarrow\quad (\varepsilon -1)\cdot k\geq 0\,,
$$
which is possible for any $\varepsilon>0$ only in case $k=0$. Thus, the inequality has the form \eqref{eq.inequal}.

Any inequality of the form \eqref{eq.inequal} which is valid for vectors $\m\in {\cal K}$ must hold with
equality for vectors in its linearity space. This implies the rest of Lemma~\ref{lem.zero-allow}.
\qed \end{proof}

A coefficient vector $\alpha\in {\dv R}^{\caP}$ satisfying \eqref{eq.o-stan}
will be called {\em o-standardized}, where $o$ stands for ``orthogonal".
Indeed, \eqref{eq.o-stan} means that $\alpha$ is in the orthogonal complement
of $L(N)$.

In the next lemma we show that the task to characterize facets of a cone of games is equivalent to the task of describing facets of the associated extended cone.

\begin{lemma}\label{lem.zero-allow2}\rm
The value of the coefficient $\alpha(\emptyset)$ in \eqref{eq.inequal}
does not influence the fact whether \eqref{eq.inequal} is facet-defining for $\m\in {\cal T}(N)$ or not.
Specifically, if \eqref{eq.inequal} is facet-defining inequality for $\m\in {\cal T}(N)$, then re-defining
the value of the coefficient $\alpha(\emptyset)$ by
$$
\alpha(\emptyset)~\coloneqq~ -\sum_{\emptyset\neq S\subseteq N} \alpha(S)
$$
yields a facet-defining inequality for both $\m\in {\cal T}(N)$ and $\m\in T(N)$. Conversely, every facet-defining inequality for\, $T(N)$ is facet-defining for ${\cal T}(N)$.
The same relation holds for the pair of cones ${\cal B}(N)$ and $B(N)$ and also for the pair of cones
${\cal E}(N)$ and $E(N)$.
\end{lemma}

\begin{proof}
First, let us discuss the question of validity of \eqref{eq.inequal} for games.
Since every $\m\in {\cal T}(N)$ satisfies $\m(\emptyset)=0$, we get
$\sum_{S\subseteq N} \alpha(S)\cdot\m(S)=\sum_{\emptyset \neq S\subseteq N} \alpha(S)\cdot\m(S)$. Thus, the value $\alpha(\emptyset)$ is irrelevant for
the validity \eqref{eq.inequal} for $\m\in {\cal T}(N)$ and, also, it does not influence
what is the face of\, ${\cal T}(N)$ specified by \eqref{eq.inequal}. Therefore, the
fact whether \eqref{eq.inequal} is facet-defining for ${\cal T}(N)$ is not influenced by
the value of $\alpha(\emptyset)$.
Thus, 
it remains to show, for any $\alpha\in {\dv R}^{\caP}$
satisfying $\alpha(\emptyset)=-\sum_{\emptyset\neq S\subseteq N} \alpha(S)$, that
\eqref{eq.inequal} is facet-defining for ${\cal T}(N)$ iff it is facet-defining for $T(N)$.

First, observe for any $\alpha\in {\dv R}^{\caP}$
satisfying $\alpha(\emptyset)=-\sum_{\emptyset\neq S\subseteq N} \alpha(S)$,
that \eqref{eq.inequal} is valid for ${\cal T}(N)$ iff is valid for $T(N)$.
By Lemma~\ref{lem.zero-allow}, its validity either for ${\cal T}(N)$ or for $T(N)$
implies that $\alpha$ is o-standardized, that is, satisfies \eqref{eq.o-stan}.
We introduce
\begin{align*}
{\cal T}_{\ell}(N) &\coloneqq
\{\, \m\in {\cal T}(N)\,:\ \m(\{i\})=0\quad \mbox{for any $i\in N$}\,\}\\
&=
\{\, \m\in T(N)\,:\ \m(S)=0\quad \mbox{for $S\subseteq N$, $|S|\leq 1$}\,\}
\end{align*}
and realize that both ${\cal T}(N)={\cal L}(N)\oplus {\cal T}_{\ell}(N)$ and
$T(N)=L(N)\oplus {\cal T}_{\ell}(N)$, where the symbol $\oplus$ denotes the direct
sum of cones. The former fact implies, for o-standardized
$\alpha\in {\dv R}^{\caP}$, that \eqref{eq.inequal} is valid for ${\cal T}(N)$
iff it is valid for $\m\in {\cal T}_{\ell}(N)$; the latter fact implies the same
for the cone $T(N)$.

To see why the claim extends to facet-defining inequalities realize that the set
$$
{\dv A}_{T} = \{\, \alpha\in {\dv R}^{\caP}\,:\
\mbox{$\alpha$ satisfies \eqref{eq.o-stan} ~and} ~~
\llangle \alpha,\m\rrangle\geq 0 ~~ \mbox{for $\m\in {\cal T}_{\ell}(N)$}\,\}
$$
is a pointed convex cone. To this end realize that, for every $A\subseteq N$, $|A|\geq 2$,
one has $\muA\in {\cal E}(N)$, with the core being the convex hull
of $\{\chi_{\{i\}}: i\in A\}$. Thus, by ${\cal E}(N)\subseteq {\cal T}(N)$, one has
$\muA\in {\cal T}(N)$, which gives $\muA\in {\cal T}_{\ell}(N)$.
Hence, any $\alpha\in {\dv A}_{T}$ satisfies $\llangle\alpha,\muA\rrangle\geq 0$ for all $A\subseteq N$, $|A|\geq 2$.
In particular, $\alpha\in -{\dv A}_{T}\cap {\dv A}_{T}$ must satisfy those inequalities with equality,
which allows one to deduce $\alpha=\b0$.

Owing to $T(N)=L(N)\oplus {\cal T}_{\ell}(N)$, the set ${\dv A}_{T}$ is the dual cone to $T(N)$.
This implies that $T(N)$ and ${\dv A}_{T}$ are polyhedral cones which are dual each other (see Section~\ref{sec.basic-notions}). Because ${\dv A}_{T}$ is pointed, it
implies that facets of\, $T(N)$ correspond to extreme rays of ${\dv A}_{T}$.

Analogously, ${\cal T}(N)={\cal L}(N)\oplus {\cal T}_{\ell}(N)$ allows one to observe that the
dual cone to ${\cal T}(N)$ has the form
${\cal I}(N)\oplus {\dv A}_{T}$, where ${\cal I}(N)$ denotes one-dimensional space
of functions $\iota\in {\dv R}^{\caP}$ such that $\iota (S)=0$ for $\emptyset\neq S\subseteq N$.
This similarly implies that facets of ${\cal T}(N)$ correspond to co-atoms of
the face-lattice of\, ${\cal I}(N)\oplus {\dv A}_{T}$, and these have the form ${\cal I}(N)\oplus R$, where
$R$ is an extreme ray of ${\dv A}_{T}$. Thus, in this context, facet-defining inequalities for either ${\cal T}(N)$
or $T(N)$ are precisely those which are given by $\alpha$'s  generating extreme rays of
${\dv A}_{T}$.

The arguments in the case of balanced cones and in the case of exact cones are analogous, so, they are left to the reader.
\qed \end{proof}

\subsection{Reflection and conjugate inequalities}\label{ssec.reflection}
Two of the extended cones are closed under a special linear self-transformation of ${\dv R}^{\caP}$.

\begin{definition}\label{def.reflection}\rm
By a {\em reflection\/} of $\m\in {\dv R}^{\caP}$ we mean $\mre\in {\dv R}^{\caP}$ given by
$$
\mre(T) \coloneqq \m(\nmT)\qquad \mbox{for any $T\subseteq N$.}
$$
Thus, reflection is nothing but inner composition with a ``complement" mapping.
\end{definition}

\begin{lemma}\label{lem.reflection}\rm
The cones $B(N)$ and $E(N)$ are closed under reflection.
\end{lemma}

\begin{proof}
To show $\m\in B(N) \Rightarrow \mre\in B(N)$ realize that one has $-C(\mg)= C(\mgr)$, where $\mg$ is
given by \eqref{eq.zero-projection}. To this end express
the core $C(\mg)$ in terms of the upper bounds instead of the lower bounds:
\begin{align*}
C(\mg) &= \{\, [x_{i}]_{i\in N} \; :\;
\mg (N)=\sum_{i\in N} x_{i}, \; \mg (S)\leq \sum_{i\in S} x_{i},\; \forall\, S\subseteq N\}\\
&= \{\, [x_{i}]_{i\in N} \; :\;
\mg (N)=\sum_{i\in N} x_{i},\;  \mg (N)- \mg (\nmT)\geq \sum_{i\in T} x_{i},\; \forall\, T\subseteq N\}\,.
\end{align*}
This allows one to write
\begin{align*}
-C(\mg) &= \{\, [y_{i}]_{i\in N} \ :\
 -\mg (N)=\sum_{i\in N} y_{i},  \mg (\nmT)-\mg (N)\leq  \sum_{i\in T} y_{i}, \forall T\subseteq N\}\\
&= \{\, [y_{i}]_{i\in N} \ :\
\mgr (N)=\sum_{i\in N} y_{i}, \; \mgr (T)\leq  \sum_{i\in T} y_{i}, \forall\, T\subseteq N\}\\ & = C(\mgr)\,,
\end{align*}
because of the relation $\mgr (T)\stackrel{\eqref{eq.zero-projection}}{=}
m^{*}(T)-m^{*}(\emptyset)= m(\nmT)-m(N) \stackrel{\eqref{eq.zero-projection}}{=}\mg (\nmT)-\mg (N)$ valid
for any $T\subseteq N$.

In order to show $\m\in E(N) \Rightarrow \mre\in E(N)$, it is enough to realize additionally that the inequality for $S\subseteq N$ is tight at
$[x_{i}]_{i\in N}\in C(\mg)$ if and only if
the inequality for $T=\nmS$ is tight at
$[y_{i}]_{i\in N}= -[x_{i}]_{i\in N}\in -C(\mg)=C(\mgr)$.
\qed \end{proof}

\begin{remark}\rm\label{rem.anti-dual}
There exists a widely used notion of duality in cooperative game theory; see \cite[Definition 6.6.3.]{PelegSudholter07}.
The \emph{dual game} of a game $\m$ is then the game $\mstar$ defined by
$$
\mstar(S) \coloneqq \m(N)-\m(N\setminus S) \quad \text{for all $S\subseteq N$.}
$$
Observe that $\mstar=\m(N)-\mre$, where $\mre$ is the reflection of $\m$ from Definition \ref{def.reflection}. Similarly, it is possible to define the notion of an \emph{anti-dual game} of $\m$ by $-\mstar$. Using \eqref{eq.zero-projection} observe that the anti-dual of $\m$ is precisely $-\mstar=\mgr$, which is
discussed in the above proof of Lemma~\ref{lem.reflection}. Then a natural question arises how main solution concepts (e.g., the core) are related to duality or anti-duality. In this context Lemma~\ref{lem.reflection} says that a game is balanced/exact if and only if its anti-dual is a balanced/exact game. The interested reader is invited to consult the paper \cite{Oishi16} for a detailed analysis of the relation of (anti-)duality and solution concepts.
\end{remark}

However, the cone $T(N)$ is not closed under reflection as the following example shows.

\begin{example}\label{exa.total-balanced}
Here is an example of $\m\in T(N)$ such that $\mre\not\in T(N)$. Put $N=\{a,b,c\}$
and $\m(N)=3$ while $\m(S)=2$ for $S\subset N$ with $|S|=2$ and $\m(R)=0$ for remaining $R\subset N$. It is a totally balanced game because it has a min-representation
by four vectors $(1,1,1)$, $(2,2,0)$, $(2,0,2)$, $(0,2,2)$. Nevertheless,
$\mgr (T)=-3$ for $T\subseteq N$, $|T|\geq 2$, and $\mgr (\{i\})=-1$ for $i\in N$.
It is a balanced game because of $C(\mgr)=\{ (-1,-1,-1)\}$ but it is not totally balanced because of
$C(\m^{\prime})=\emptyset$ for the restriction $\m^{\prime}$ of $\mgr$ to subsets of $\{a,b\}$.
\end{example}

A concept related to the reflection is the following one.

\begin{definition}\label{def.conjugate-ineq}\rm
Every inequality \eqref{eq.inequal} is assigned a {\em conjugate inequality} of the form
\begin{equation}
\underbrace{\sum_{T\subseteq N} \alpha^{*} (T)\cdot \m(T)}_{=\llangle \alpha^{*},\m\rrangle}\geq 0\qquad
\mbox{required for $\m\in {\dv R}^{\caP}$}\,,
\label{eq.conjug-ineq}
\end{equation}
where $\alpha^{*}\in {\dv R}^{\caP}$ is the reflection of the coefficient vector $\alpha\in {\dv R}^{\caP}$ in \eqref{eq.inequal}.
\end{definition}

Note in this context that $\alpha$ is o-standardized iff $\alpha^{*}$ is o-standardized.
An important fact appears to be the formula
\begin{equation}
\llangle \alpha^{*},m\rrangle =
\sum_{T\subseteq N} \alpha^{*}(T)\cdot \m(T) 
\stackrel{T\coloneqq\nmS}{=}
\sum_{S\subseteq N} \alpha(S)\cdot \mre(S)
= \llangle \alpha,\mre\rrangle
\label{eq.conjug-relect}
\end{equation}
valid for any pair $\alpha,\m\in {\dv R}^{\caP}$.
It enables us to prove the following statement.

\begin{lemma}\rm\label{lem.facet-conjugate}
An inequality \eqref{eq.inequal} is facet-defining for $E(N)$ iff its
conjugate inequality \eqref{eq.conjug-ineq} is facet-defining for $E(N)$.
The same relation holds for the cone $B(N)$.
\end{lemma}

\begin{proof}
We first show that \eqref{eq.inequal} is valid for $E(N)$ iff \eqref{eq.conjug-ineq} is valid for
$E(N)$. Because of $\alpha^{**}=\alpha$ it is enough to verify that the validity of \eqref{eq.inequal} implies
the validity of \eqref{eq.conjug-ineq}. Given $\m\in E(N)$, one has $\mre\in E(N)$ by Lemma~\ref{lem.reflection} and
$\llangle \alpha^{*},m\rrangle \stackrel{\eqref{eq.conjug-relect}}{=} \llangle \alpha,\mre\rrangle \stackrel{\eqref{eq.inequal}}{\geq} 0$.

One can perform an analogous consideration as in the proof of Lemma~\ref{lem.zero-allow2}, that is, to \mbox{introduce}
${\cal E}_{\ell}(N) \coloneqq \{\, \m\in E(N)\,:\ \m(S)=0\quad \mbox{for $S\subseteq N$, $|S|\leq 1$}\,\}$ and put
$$
{\dv A}_{E} ~\coloneqq~ \{\, \alpha\in {\dv R}^{\caP}\,:\
\mbox{$\alpha$ satisfied \eqref{eq.o-stan} ~and} ~~
\llangle \alpha,\m\rrangle\geq 0 ~~ \mbox{for $\m\in {\cal E}_{\ell}(N)$}\,\}\,.
$$
The fact $E(N)=L(N)\oplus {\cal E}_{\ell}(N)$ then implies that ${\dv A}_{E}$ is the dual cone to $E(N)$; it means that $\alpha\in {\dv A}_{E}$ iff \eqref{eq.inequal} holds for any $\m\in E(N)$. This implies the validity
of \eqref{eq.conjug-ineq} for any $\m\in E(N)$, which means that ${\dv A}_{E}$ is closed under reflection. In particular,
the reflection is a one-to-one linear mapping from ${\dv R}^{\caP}$ to ${\dv R}^{\caP}$ which transforms
${\dv A}_{E}$ onto itself. Such a linear mapping transforms faces of ${\dv A}_{E}$ to faces of ${\dv A}_{E}$ of the same dimension.

One can show that ${\dv A}_{E}$ is a pointed cone by the same arguments as in the proof of Lemma~\ref{lem.zero-allow2}. As the cones $E(N)$ and ${\dv A}_{E}$ are dual each other
the facets of\, $E(N)$ correspond to the extreme rays of ${\dv A}_{E}$.
Since the coefficients of facet-defining inequalities for $E(N)$ are just those which generate extreme rays
of ${\dv A}_{E}$ and these are mapped by the reflection to (other) extreme rays
of ${\dv A}_{E}$, the conjugate inequalities to them must also be facet-defining.

The arguments for the balanced cone $B(N)$ are analogous and are left to the reader.
\qed \end{proof}

In fact, the statement from Lemma \ref{lem.facet-conjugate} holds true for any full-dimensional cone in ${\dv R}^{\caP}$ that is closed under reflection.

\subsection{How to assign an inequality to a min-balanced system}\label{ssec.min-balan-ineq}
Facet-defining inequalities \eqref{eq.inequal} for cones $B(N)$ and $T(N)$ appear to be
determined uniquely (up to a positive multiple) by the induced set systems
$$
{\cal B}_{\alpha} \coloneqq \{\,S\subseteq N\,:\ \alpha(S)<0\,\},
$$
where $\alpha\in {\dv R}^{\caP}$ is the respective coefficient vector. A converse relation is established in this section.
More specifically, given a {\em non-trivial min-balanced system $\calB$}, by Lemma~\ref{lem.equiv-min-bal},
there are unique coefficients
$\lambda_{S}$, $S\in\calB$, with
$$
\chi_{M} =\sum_{S\in\calB} \lambda_{S}\cdot \chi_{S}\qquad \mbox{where $M=\bigcup\calB$ and $\lambda_{S}>0$ for $S\in\calB$.}
$$
In fact, one can even show that $\lambda_{S}\in {\dv Q}$. Indeed,
$\chi_{M} =\sum_{S\in\calB} \lambda_{S}\cdot \chi_{S}$ means that the coefficient
vector $\lambda\in {\dv R}^{\calB}$ is a solution of a matrix equality
$\lambda\cdot C=\chi_{M}$ with a zero-one matrix $C\in {\dv R}^{\calB\times N}$.
Since unique solution exists, a regular column $\calB\times T$-submatrix  of $C$,
where $T\subseteq N$, $|T|=|\calB|$, exists such that $\lambda\cdot C^{\calB\times T}=\chi_{M\cap T}$.
Since $C$ has zero columns for $i\in \nmM$ one has $T\subseteq M$.
Nevertheless, the inverse of this regular zero-one submatrix is
a rational matrix, which implies that the components of $\lambda$ are in ${\dv Q}$.

Thus, a unique integer $k\geq1$ exists such that all
$k\cdot\lambda_{S}\in {\dv Z}$ with $S\in\calB$ are relatively prime. For any $S\subseteq N$, define
\begin{equation}
\alpha_{\calB}(S) ~\coloneqq~
\begin{cases}
\,\,k & \mbox{if $S=M=\bigcup\calB$},\\
-k\cdot\lambda_{S} & \mbox{if $S\in\calB$},\\
-k+k\sum\limits_{S\in\calB}\lambda_S =-\alpha_{\calB}(M)-\sum\limits_{S\in\calB}\alpha_{\calB}(S)& \mbox{if $S=\emptyset$},\\
\,\,0 & \text{otherwise.}
\end{cases}~~~~
\label{eq.coef-def}
\end{equation}
This ensures the following equality:
\begin{equation}
\sum_{S\subseteq N}\alpha_{\calB} (S)\cdot\chi_{S} =\alpha_{\calB}(M)\cdot\chi_{M} +\sum_{S\in\calB} \alpha_{\calB}(S)\cdot\chi_{S} +\alpha_{\calB}(\emptyset)\cdot\underbrace{\chi_{\emptyset}}_{=\mbox{\scriptsize\sf 0}} =\0\in {\dv R}^{N}\,.
\label{eq.B-decom}
\end{equation}
For any $i\in M$ one has
\begin{align*}
\sum_{S\subseteq N:\, i\in S} \alpha_{\calB} (S)&=\sum_{S\in\{M\}\cup\calB:\, i\in S} \alpha_{\calB} (S)=
k- k\cdot\sum_{S\in\calB:\, i\in S} \lambda_{S}=\\ & =k\cdot (1-\sum_{S\in\calB:\, i\in S} \lambda_{S}) =
k\cdot 0= 0\,.
\end{align*}
Analogously, by the definition of $M$, one has
$\sum_{S\subseteq N:\, i\in S} \alpha_{\calB} (S)=0$ for any $i\in \nmM$; thus, the vector $\alpha_{\calB}\in {\dv Z}^{\caP}$
is o-standardized. Choose $i\in M$ and consider $T\in\calB$ such that $i\not\in T$ (note that $\calB$ is non-trivial),
which allows one to write
\begin{align*}
\alpha_{\calB}(\emptyset) & = -\alpha_{\calB}(M)-\sum_{S\in\calB} \alpha_{\calB}(S)=
-k +k\cdot\sum_{S\in\calB} \lambda_{S}=k\cdot (\sum_{S\in\calB} \lambda_{S}-1)\\
&= k\cdot \underbrace{(\sum_{S\in\calB:\, i\in S} \lambda_{S}-1)}_{=0} ~+~
k\cdot \underbrace{(\sum_{S\in\calB:\, i\not\in S} \lambda_{S})}_{\geq \lambda_{T}} > 0,
\end{align*}
which implies $\alpha_{\calB} (\emptyset)\geq 1$. Thus, the corresponding inequality
\begin{equation}
\llangle \alpha_{\calB},\m\rrangle =\alpha_{\calB}(M)\cdot \m(M) +\sum_{S\in\calB} \alpha_{\calB}(S)\cdot \m(S)
+\alpha_{\calB}(\emptyset)\cdot \m(\emptyset)\geq 0\quad \mbox{for $\m\in {\dv R}^{\caP}$}
\label{eq.B-ineq}
\end{equation}
of the form \eqref{eq.inequal} is assigned the set system $\calB=\calB_{\alpha}$ with $\alpha=\alpha_{\calB}$.
This yields mutually inverse transformation $\calB\leftrightarrow\alpha_{\calB}=\alpha$
between non-trivial min-balanced systems and the coefficient vectors in the corresponding inequalities.

\begin{definition}\label{def.min-bal-coef}\rm
Given a non-trivial min-balanced system $\calB\subseteq\caP$, the above coefficient vector
in ${\dv Z}^{\caP}$ defined in \eqref{eq.coef-def} will be denoted by $\alpha_{\calB}$.
\end{definition}

\begin{example}
Consider the min-balanced system $\calB$ from  Example \ref{exa.minbal}
with the player set $N=\{a,b,c,d,e\}$, that is, $\calB = \{\{a,b\},\{a,c\},\{a,d\},\{b,c,d\}\}$. The carrier of $\calB$ is $M=\{a,b,c,d\}$ and the coefficient vector is
$$
\alpha_{\calB}(S) \coloneqq
\begin{cases}
\,\,3 & \mbox{for $S=M$}, \\
-1 & \mbox{for $S\in \{\{a,b\},\{a,c\},\{a,d\}\}$}, \\
-2 & \mbox{for $S=\{b,c,d\}$},\\
\,\,2 & \mbox{for $S=\emptyset$},\\
\,\,0 & \text{otherwise.}
\end{cases}
$$
Then the inequality \eqref{eq.B-ineq} takes the form
$$
3\cdot\m(\{a,b,c,d\}) - \m(\{a,b\})-\m(\{a,c\})-\m(\{a,d\}) - 2\cdot\m(\{b,c,d\}) + 2\cdot\m(\emptyset) \geq 0\,.
$$
\end{example}

The next lemma follows from the results in \cite[Theorem\,2]{Shapley1967On-balanced-set};
one just applies what is said in Section~\ref{ssec.zero-allow} about the correspondence between cones $B(N)$ and ${\cal B}(N)$.

\begin{lemma}\label{lem.balanced-cone}\rm
Assuming $|N|\geq 2$ the facet-defining inequalities for $B(N)$ are just those
which correspond to non-trivial min-balanced systems\/ $\calB$ on $N$.
In particular, given $\m\in {\dv R}^{\caP}$,
$$
\m\in B(N)~\Leftrightarrow~\llangle\alpha_{\calB},\m\rrangle\geq 0 \quad \text{for any non-trivial min-balanced system $\calB$ on $N$.}
$$
\end{lemma}

The following concept relates the above observation to concepts from Section~\ref{ssec.reflection}.

\begin{definition}\label{def.complementary-syst}\rm
Every min-balanced system $\calB\subseteq {\cal P}(N)$ is assigned its {\em complementary system\/}
$$
\calB^{*}\coloneqq \{\nmS \,:\ S\in\calB\}\,.
$$
\end{definition}

The fact that $B(N)$ is closed under reflection 
then basically implies the following.

\begin{corollary}\label{cor.balanced-dual-N}\rm
If $\calB$ is a non-trivial min-balanced set system on $N$, then its complementary system
${\cal B}^{*}$ is also a non-trivial min-balanced system on $N$ inducing the
conjugate inequality to the inequality \eqref{eq.B-ineq} with $M=N$.
\end{corollary}

\begin{proof}
If $\calB$ is a min-balanced set system on $N$, then write using \eqref{eq.B-decom} and \eqref{eq.coef-def}:
$$
\0=k\cdot\chi_{N}-\sum_{S\in\calB}k\cdot\lambda_{S}\cdot\chi_{S}
+(-k+\sum_{S\in\calB}k\cdot\lambda_{S})\cdot\underbrace{\chi_{\emptyset}}_{=\mbox{\scriptsize\sf 0}}\,,
$$
where the coefficients $\lambda_{S}$ and the constant $k$ are introduced in the
beginning of Section~\ref{ssec.min-balan-ineq}. We omit the zero term containing $\chi_{\emptyset}$, multiply the equality by $(-1)$ and extend the right-hand side of it by $\pm k\cdot\sum_{S\in\calB}\lambda_{S}\cdot\chi_{N}$ to get
\begin{eqnarray*}
\0 &=&(-k+k\cdot\sum_{S\in\calB}\lambda_{S})\cdot\chi_{N} -\sum_{S\in\calB}k\cdot\lambda_{S}\cdot\chi_{\nmS} \\
&=& (-k+\sum_{T\in\calB^{*}}k\cdot\lambda_{\nmT})\cdot\chi_{N} -\sum_{T\in\calB^{*}}k\cdot\lambda_{\nmT}\cdot\chi_{T} + k\cdot\underbrace{\chi_{\emptyset}}_{=\mbox{\scriptsize\sf 0}},
\end{eqnarray*}
because the coefficient with $\chi_{\emptyset}=\0$ is immaterial. Thus, the coefficient vectors assigned to $\calB$ and $\calB^{*}$ by \eqref{eq.coef-def} are related by $\alpha_{\calB}^{*}=\alpha_{\calB^{*}}$.

By Lemma~\ref{lem.balanced-cone}, $\llangle \alpha_{\calB},\m\rrangle\geq 0$ is facet-defining
inequality for $\m\in B(N)$. By (the second claim of) Lemma~\ref{lem.facet-conjugate}
$\llangle \alpha_{\calB^{*}},\m\rrangle=\llangle \alpha_{\calB}^{*},\m\rrangle\geq 0$
is also facet-defining inequality for $\m\in B(N)$. In particular, by Lemma~\ref{lem.balanced-cone},
$\alpha_{\calB^{*}}$ must be a positive multiple of $\alpha_{\cal C}$ for a non-trivial
min-balanced system ${\cal C}$ on $N$. By \eqref{eq.coef-def}, this gives $\calB^{*}={\cal C}$ and
$\calB^{*}$ is min-balanced. The fact that $k\cdot\lambda_{S}$, $S\in \calB$, are relatively prime
integers then implies $\alpha_{\calB^{*}}=\alpha_{\cal C}$. Thus, $\llangle \alpha_{\cal C},\m\rrangle\geq 0$
is a conjugate inequality to \eqref{eq.B-ineq} (with $M=N$).
\qed \end{proof}

Note that the above observation was already mentioned by Shapley \cite{Shapley1967On-balanced-set} as an empirical fact observed in case $|N|\leq 5$.

\section{Irreducible systems}\label{sec.irreducible}
We are in position to introduce the following concept, which is crucial for our main result.

\begin{definition}\label{def.irreducible}\rm
We say that a min-balanced system $\calB\subseteq\caP$ is
{\em reducible\/} if there exists $A\subset \bigcup\calB$ and
$B\in\calB_{A}\coloneqq\{ S\in\calB\,:\ S\subset A\}$ such that
\begin{enumerate}
\item $\chi_{A}$ is in the conic hull of $\{\chi_{S}\,:\, S\in\calB_{A}\}$ and
\item $\chi_{\bigcup\calB}$  is in the conic hull of $\{\chi_{T}\,:\, T\in\{A\}\cup(\calB\setminus \{B\})\}$.
\end{enumerate}
A min-balanced system $\calB\subseteq\caP$ that is not reducible is called {\em irreducible\/}.
\end{definition}

A reducible min-balanced system is necessarily non-trivial.
Note that, without loss of generality, one can require
$A=\bigcup\calB_{A}\not\in\calB$ in the definition of reducibility, as proved below
in Lemma~\ref{lem.irreducible}.
The meaning of the reducibility condition is that the inequality corresponding to
$\calB$ is derivable from the inequalities which correspond to other min-balanced systems
$\calB^{\prime}$ satisfying $\bigcup\calB^{\prime}\subseteq\bigcup\calB$.

Clearly, the concept of irreducibility is invariant with respect to a permutation of players.
The types of non-trivial irreducible min-balanced system on $N$ for $2\leq |N|\leq 4$
can be found in Appendix.

\begin{lemma}\label{lem.irreducible}\rm
Let $\calB$ be a min-balanced system. If  $A\subset \bigcup\calB$
and $B\in\calB_{A}$ exist such that both $\calB_{A}$ and $\{A\}\cup (\calB\setminus\{B\})$
satisfy the conditions from Definition~\ref{def.irreducible}, then
$|A|\geq 2$ and one has both $\bigcup\calB_{A}=A$ and $A\not\in\calB$.\\
Moreover, there exist a min-balanced system $\calC\subseteq\calB_{A}$ on $A$ with $B\in\calC$ and a min-balanced system $\calD\subseteq \{A\}\cup (\calB\setminus\{B\})$ on $\bigcup\calB$ with $A\in\calD$ such that
the inequality \eqref{eq.B-ineq} corresponding to $\calB$ is a conic combination of inequalities corresponding to  $\calC$ and $\calD$.
\end{lemma}

\begin{proof}
Throughout the proof we will use $M$ to denote the carrier of $\calB$, that is,
$M\coloneqq \bigcup\calB$.
As $\calB$ is non-trivial, one can apply Lemma~\ref{lem.min-balanced}. Hence,  $B\in\calB$ implies $B\neq\emptyset$ and $B\subset A$ gives $|A|\geq 2$.
The assumption that $\chi_{A}$ is in the conic hull of $\{\chi_{S}\,:\, S\in\calB_{A}\}$
means that there exist $\mu_{S}\geq 0$, $S\in\calB_{A}$, with
$\chi_{A}=\sum_{S\in\calB_{A}} \mu_{S}\cdot\chi_{S}$, which implies $A=\bigcup\calB_{A}$.

Analogously, the other assumption says that there exist $\beta_{T}\geq 0$,
$T\in\{A\}\cup (\calB\setminus\{B\})$, with
$\chi_{M}=\sum_{T\in\{A\}\cup (\calB\setminus\{B\})} \beta_{T}\cdot\chi_{T}$. Thus,
\begin{align*}
\chi_{M} &= \beta_{A}\cdot\chi_{A}\, + \sum_{T\in\calB\setminus\{A,B\}} \beta_{T}\cdot\chi_{T} ~=~
\beta_{A}\cdot \sum_{S\in\calB_{A}} \mu_{S}\cdot\chi_{S}\, + \sum_{T\in\calB\setminus\{A,B\}} \beta_{T}\cdot\chi_{T}\\
&= \sum_{S\in\calB_{A}} \beta_{A}\cdot\mu_{S}\cdot\chi_{S}\, + \sum_{T\in\calB\setminus\{A,B\}} \beta_{T}\cdot\chi_{T}\\
&= \beta_{A}\cdot\mu_{B}\cdot\chi_{B}\, + \sum_{S\in\calB_{A}\setminus\{B\}} (\beta_{A}\cdot\mu_{S}+\beta_{S})\cdot\chi_{S}\,
+ \sum_{T\in\calB\setminus(\{A\}\cup\calB_{A})} \beta_{T}\cdot\chi_{T}\,.
\end{align*}
On the other hand, since $\calB$ is min-balanced, the uniquely
determined coefficients $\lambda_{S}$, $S\in\calB$, in the decomposition
\begin{equation}
\chi_{M} = \sum_{S\in\calB} \lambda_{S}\cdot\chi_{S}
\label{eq.decomp}
\end{equation}
must be all strictly positive by Lemma \ref{lem.equiv-min-bal}. The uniqueness of the coefficients implies
\begin{align}
\lambda_{B} &= \beta_{A}\cdot\mu_{B} \label{eq.lambda-decomp}\\
\lambda_{S} &= \beta_{A}\cdot\mu_{S}+\beta_{S}\quad \mbox{for $S\in\calB_{A}\setminus\{B\}$,} \nonumber\\
\lambda_{T} &= \beta_{T}\qquad\qquad\quad \mbox{~for $T\in \calB\setminus(\{A\}\cup\calB_{A})$,~ and}  \nonumber \\
\lambda_{R} &=  0 \qquad\qquad\qquad \mbox{for other $R\subseteq N$.} \nonumber
\end{align}
This implies $A\not\in\calB$ as otherwise a contradictory conclusion $\lambda_{A}>0$ and $\lambda_{A}=0$  is derived from \eqref{eq.decomp} and  \eqref{eq.lambda-decomp}.

As $\lambda_{S}>0$ for $S\in\calB$, the condition \eqref{eq.lambda-decomp}
gives $\lambda_{B}>0 ~\Rightarrow~ \beta_{A}>0$ and $\mu_{B}>0$. We put
$$
\calC \coloneqq \{S\in\calB_{A} \,:\, \mu_{S}>0\} \quad\mbox{and}\quad
\calD \coloneqq \{T\in\{A\}\cup (\calB\setminus \{B\})\,:\, \beta_{T}>0\}
$$
and obtain $B\in\calC$ and $A\in\calD$. Since $\calC\subseteq\calB_{A}\subseteq\calB$ the linear independence condition (ii) from Lemma~\ref{lem.equiv-min-bal} for $\calC$ is evident and, therefore, again by Lemma~\ref{lem.equiv-min-bal}, $\calC$ is min-balanced on $A$.
To verify the linear independence condition (ii) for $\{A\}\cup (\calB\setminus\{B\})$ it is enough to show that $\chi_{A}$ is {\em not\/} in the linear hull of
$\{\chi_{S}\,:\ S\in\calB\setminus\{B\}\,\}$.
Thus, assume for contradiction that $\chi_{A}=\sum_{S\in\calB\setminus\{B\}} \gamma_{S}\cdot\chi_{S}$ with real $\gamma_{S}$, $S\in\calB\setminus\{B\}$, which yields
\begin{align*}
\chi_{M} &= \beta_{A}\cdot\chi_{A} + \sum_{T\in\calB\setminus\{B\}} \beta_{T}\cdot\chi_{T}=
\beta_{A}\cdot (\sum_{S\in\calB\setminus\{B\}} \gamma_{S}\cdot\chi_{S}) + \sum_{T\in\calB\setminus\{B\}} \beta_{T}\cdot\chi_{T}\\
&= \sum_{T\in\calB\setminus\{B\}} (\beta_{A}\cdot\gamma_{T}+\beta_{T})\cdot\chi_{T}\,.
\end{align*}
As $\calB$ is min-balanced on $M$, the uniqueness of the coefficients in \eqref{eq.decomp} implies a contradictory conclusion that
simultaneously $\lambda_{B}>0$ by \eqref{eq.decomp} and $\lambda_{B}=0$. Thus, $\calD\subseteq \{A\}\cup(\calB\setminus \{B\})$ satisfies both (i) and (ii), so, it is min-balanced on $M$ by Lemma~\ref{lem.equiv-min-bal}.

In order to prove the last claim, apply \eqref{eq.lambda-decomp} to express the inequality
$\m(M)-\sum_{S\in\calB} \lambda_{S}\cdot \m(S)\geq 0$
for $\m\in {\dv R}^{\caP\setminus\{\emptyset\}}$ and non-empty components only
as the sum of inequalities
\begin{align*}
\beta_{A}\cdot \m(A)- \sum_{S\in\calB_{A}} \beta_{A}\cdot\mu_{S}\cdot \m(S) ~=~
\beta_{A}\cdot [\m(A)- \sum_{S\in\calC} \mu_{S}\cdot \m(S)] &\geq 0,\\
\m(M)-\beta_{A}\cdot \m(A)-\sum_{S\in\calB\setminus\{B\}} \beta_{S}\cdot \m(S) ~=~
\m(M)-\sum_{T\in\calD} \beta_{T}\cdot \m(T) &\geq 0.
\end{align*}
As explained in Section \ref{ssec.min-balan-ineq}, the standard inequalities for the
min-balanced systems are positive multiples of these, which implies the
result in ${\dv R}^{\caP\setminus\{\emptyset\}}$. The coefficients with the
empty set are immaterial since they are determined by the o-standardization
condition.
\qed \end{proof}

Let us give a simple example of a reducible system, which illustrates Lemma~\ref{lem.irreducible}.

\begin{example}\label{exa.reduction}
Put $N=\{a,b,c,d\}$. Let $\calB = \{\, \{a\},\{b\},\{c\}\, \}$ be a set system on a strict subset $M=\{a,b,c\}$ of $N$. The corresponding inequality
\eqref{eq.B-ineq} is
$$
\m (\{a,b,c\}) -\m (\{a\})  -\m (\{b\})  -\m (\{c\}) +2\cdot\m (\emptyset) ~\geq~ 0\,.
$$
Take $A=\{a,b\}$ and observe that $\chi_{\{a,b\}}=\chi_{\{a\}}+\chi_{\{b\}}$, which gives
$\chi_{A}\in\mbox{cone}\,\{\chi_{S}\,:\, S\in\calB_{A}\}$ with $\calB_{A}=\{\,\{a\},\{b\}\,\}$.
We also have $\chi_{\{a,b,c\}}=\chi_{\{a,b\}}+\chi_{\{c\}}$ which implies that
$$
\chi_{M}\in\mbox{cone}\,\{\chi_{T}\,:\, T\in\{A\}\cup (\calB\setminus\calB_{A}) \}\subseteq
\mbox{cone}\,\{\chi_{T}\,:\, T\in\{A\}\cup (\calB\setminus\{B\}) \}
$$
for both $B\in\calB_{A}=\{\, \{a\},\{b\}\, \}$. One has $\calC=\calB_{A}$ and, for both $B=\{a\}$ and $B=\{b\}$,
the respective $\calD$ is $\{A\}\cup (\calB\setminus\calB_{A})=\{\, \{a,b\},\{c\}\,\}$.
The inequalities \eqref{eq.B-ineq} for $\calC$ and $\calD$ are
\begin{align*}
\m (\{a,b\}) -\m (\{a\})  -\m (\{b\})  +\m (\emptyset) &\geq 0\,,\\
\m (\{a,b,c\}) -\m (\{a,b\})  -\m (\{c\}) +\m (\emptyset) &\geq 0\,.
\end{align*}
Their sum is the above inequality \eqref{eq.B-ineq} corresponding to ${\calB}$.
\end{example}
\noindent

A more complicated example of a reducible min-balanced system is as follows.

\begin{example}
Put $N=\{a,b,c,d,e\}$ and consider a set system
\[
\calB = \{\,\{a,b\},\{c,e\},\{d,e\},\{a,c,d\},\{b,c,d\}\,\}.
\]
It is easy to show that $\calB$ is min-balanced on $N$ using Lemma~\ref{lem.equiv-min-bal}.
To show that $\calB$ is reducible according to Definition \ref{def.irreducible} we take $A\coloneqq \{a,b,c,d\}$ and observe that
\[
\calB_A = \{\,\{a,b\},\{a,c,d\},\{b,c,d\}\,\}.
\]
It follows that $\chi_A$ is in the conic hull of  $\{\chi_{S}\,:\, S\in\calB_{A}\}$ because of
$$
\chi_A = \frac 12\cdot (\chi_{\{a,b\}} + \chi_{\{a,c,d\}} + \chi_{\{b,c,d\}}).
$$
Consider, for example, $B\coloneqq \{a,c,d\}\in \calB_A$.
Also $\chi_{\bigcup\calB}=\chi_N$  is in the conic hull of $\{\chi_{T}\,:\, T\in\{A\}\cup(\calB\setminus \{B\})\}$
since
\[
\chi_N = \frac 12\cdot (\chi_{\{a,b,c,d\}} +\chi_{\{a,b\}} + \chi_{\{c,e\}} + \chi_{\{d,e\}}).
\]
Hence, $\calB$ is reducible.
\end{example}

The next example shows that some sets in an irreducible system
can be inclusion comparable.

\begin{example}
Put $N=\{a,b,c,d,e\}$ and consider a set system
$$
\calB = \{\,\{a,b\},\{a,c,d\},\{a,c,e\},\{a,b,d,e\},\{b,c,d,e\}\,\}\,.
$$
Then one has
$$
\chi_N = \frac{1}{4}\cdot (\chi_{\{a,b\}}  + \chi_{\{a,c,d\}} + \chi_{\{a,c,e\}} + \chi_{\{a,b,d,e\}})
+\frac{1}{2}\cdot \chi_{\{b,c,d,e\}}\,,
$$
which allows one to show, using Lemma~\ref{lem.equiv-min-bal}, that $\calB$ is min-balanced on $N$.

Owing to (the first claim in) Lemma~\ref{lem.irreducible} it is enough to test sets $A\subset\bigcup\calB=N$
with $|A|\geq 2$ and $A\not\in\calB$ whether $A=\bigcup\calB_{A}$ with $\calB_{A}=\{S\in\calB\,:\ B\subset A\}$ and the conditions from Definition \ref{def.irreducible} hold.
First observe that there is no set $A\subset N$ with $2\leq |A|\leq 3$ satisfying $A=\bigcup\calB_{A}$.
Second, consider $A\subset N$ with $|A|=4$ and $A\not\in\calB$ and distinguish three cases:
\begin{enumerate}
\item If $A=\{a,b,c,d\}$ then $\calB_{A}=\{\,\{a,b\},\{a,c,d\}\,\}$ and one has $A=\bigcup\calB_{A}$; however,
the vector $\chi_A$ is not in the conic hull of $\{\chi_{S}\,:\, S\in\calB_{A}\}$.
\item If $A=\{a,b,c,e\}$ then $\calB_{A}=\{\,\{a,b\},\{a,c,e\}\,\}$ and the arguments are same
as in the previous case.
\item If $A=\{a,c,d,e\}$ then $\calB_{A}=\{\,\{a,c,d\},\{a,c,e\}\,\}$ and, again, $\chi_A$ is not in
the conic hull of $\{\chi_{S}\,:\, S\in\calB_{A}\}$.
\end{enumerate}
In any case, the first condition in Definition \ref{def.irreducible} not valid for any such set $A$,
which implies that the min-balanced system $\calB$ is irreducible.
\end{example}

The following result is a direct consequence of Lemma~\ref{lem.irreducible}.
\begin{corollary}\label{cor.irreducible}\rm
Given a reducible min-balanced system $\calB$, the corresponding inequality is a conic combination of
inequalities which correspond to other min-balanced systems $\calB^{\prime}$ with $\bigcup\calB^{\prime}\subseteq\bigcup\calB$.
In particular, it is a combination of inequalities which correspond to irreducible min-balanced systems on subsets of $\bigcup\calB$.
\end{corollary}

\section{Main result}\label{sec.main-result}
The main result of this paper, proved in Section \ref{ssec.basic-total}, is as follows.

\begin{theorem}\label{thm.total}
Let $|N|\geq 2$. The facet-defining inequalities for\/ $T(N)$ are exactly those
which correspond to non-trivial irreducible min-balanced systems $\calB$ with $\bigcup\calB\subseteq N$.
\end{theorem}

Note that $T(N)$ is not closed under reflection (Example~\ref{exa.total-balanced}); thus, one cannot expect
that the conjugate inequality to a facet-defining inequality for $T(N)$ is also facet-defining.

\subsection{Characterization of the balanced and exact cones}\label{sec.dual-balanced}
We are going to show that the sets $B(N)$ and $E(N)$ are rational polyhedral cones, by simplifying the arguments from \cite[Section 3]{Csoka2011Balancedness}.
The following notation will be instrumental.

\begin{definition}\rm\label{def.Theta-cones}
Assume $|N|\geq 2$. For any non-empty set $D\subseteq N$, we put
\begin{align*}
\Theta^{N}_{D} \coloneqq \{\,\, \theta\in {\dv R}^{\caP}\, :\
& \theta (S)\leq 0 ~~\mbox{for any $S\subseteq N$,  $S\not\in\{\emptyset, D, N\}$},\\
&\sum_{T\subseteq N} \theta (T)=0 ~~~\mbox{and}~ \sum_{T\subseteq N:\, i\in T} \theta (T)=0 ~~\mbox{for any $i\in N$}\,\}.
\end{align*}
\end{definition}

In words, $\Theta^{N}_{D}$ is the set of o-standardized  vectors which are non-positive outside $\{\emptyset, D, N\}$.

\begin{lemma}\label{lem.dual-cones}\rm
If $|N|\geq 2$ and $\emptyset\neq D\subseteq N$, then  $\Theta^{N}_{N}\subseteq \Theta^{N}_{D}$
and $\theta(N),\theta(\emptyset)\geq 0$, for any $\theta\in\Theta^{N}_{D}$.
The only $\theta\in\Theta^{N}_{D}$ satisfying $\theta(N)+\theta(\emptyset)\geq 0$ with equality is
the zero game $\theta=\b0$.
In particular, $\Theta^{N}_{D}$ is a pointed polyhedral cone containing $\Theta^{N}_{N}$ and every non-zero $\theta\in\Theta^{N}_{N}$ satisfies both $\theta(N)>0$ and $\theta(\emptyset)>0$.
\end{lemma}

\begin{proof}
The inclusion $\Theta^{N}_{N}\subseteq \Theta^{N}_{D}$ is evident. Thus,
assume without loss of generality that $D\subset N$. Then $\theta (N)\geq 0$
is a valid inequality for $\theta\in\Theta^{N}_{D}$. Indeed, take $i\in \nmD$ and realize that
$\theta(N)=-\sum_{T\subset N:\, i\in T} \theta (T)\geq 0$.
Analogously, for any $j\in N$,
$$
\sum_{S\subseteq N\smin\{j\}} \theta (S)=\sum_{S\subseteq N} \theta (S)-\sum_{T\subseteq N:\, j\in T} \theta (T)=0-0=0\,,
$$
which implies that $\theta (\emptyset)\geq 0$
is a valid inequality for $\theta\in\Theta^{N}_{D}$; it suffices to take $j\in D$ and write
$\theta(\emptyset)=-\sum_{\emptyset\neq S\subseteq N\smin\{j\}} \theta (S)\geq 0$.
Thus, $\theta(N)+\theta(\emptyset)\geq 0$ is a valid inequality for $\theta\in\Theta^{N}_{D}$.

To show that the only $\theta\in\Theta^{N}_{D}$ satisfying $\theta(N)+\theta(\emptyset)=0$ is $\theta =\b0$,
realize that both $\theta(N)=0$ and $\theta(\emptyset)=0$. The identity $\theta(N)=0$
implies by $0=\theta(N)=-\sum_{T\subset N:\, i\in T} \theta (T)\geq 0$, for $i\in \nmD$,
that $\theta(T)$ vanishes for any $T\subseteq N$ intersecting $\nmD$.
Analogously, the identity $\theta(\emptyset)=0$ implies by
$0=\theta(\emptyset)=-\sum_{\emptyset\neq S\subseteq N\smin\{j\}} \theta (S)\geq 0$,
for $j\in D$, that $\theta(S)$ vanishes for any $S\subset D$. Hence,
$0=\sum_{T\subseteq N} \theta (T)=\theta(D)$ and $\theta=\b0$.

The proof of the fact that the only $\theta\in\Theta^{N}_{N}$ satisfying $\theta(N)\geq 0$ with
equality is $\theta =\b0$ is analogous: for any $i\in N$ we get
$0=\theta(N)=-\sum_{T\subset N:\, i\in T} \theta (T)\geq 0$. Hence, $\theta(T)=0$ for any $\emptyset\neq T\subseteq N$
and $\theta(\emptyset)=-\sum_{\emptyset\neq T\subseteq N} \theta (T)=0$. The same with
$\theta\in\Theta^{N}_{N}$ satisfying $\theta(\emptyset)=0$: take $j\in N$ and get
$0=\theta(\emptyset)=-\sum_{\emptyset\neq S\subseteq N\smin\{j\}} \theta (S)\geq 0$,
which implies that $\theta(S)$ vanishes for $S\subset N$ and, consequently,
$\theta(N)=-\sum_{S\subset N} \theta (S)=0$. Thus, every
non-zero $\theta\in\Theta^{N}_{N}$ satisfies both $\theta(N)>0$ and $\theta(\emptyset)>0$.
\qed \end{proof}

To characterize the sets $B(N)$ and $E(N)$, we are going to use the following
criterion for feasibility of a system of linear constraints involving both inequalities and equalities.
In the statement, $C_{I}$ denotes the submatrix of a matrix $C$ specified by a set of rows $I$,
$b_{I}$ is the subvector of a vector $b$ determined by its components in $I$, $\0_{I}$ is  the zero vector in ${\dv R}^{I}$,
and $C^{\top}$/$b^{\top}$ denotes the transpose of a matrix $C$/vector $b$. Inequalities and equalities for vectors are
understood coordinatewise.

\begin{lemma}\label{lem.feasibility-criterion}\rm
Let $C\in {\dv R}^{L\times N}$ be a real matrix and $b\in {\dv R}^{L}$ a real column vector, where
$L=I\cup E$ with disjoint sets $I$ and $E$. Then the condition
$$
\exists\, x\in {\dv R}^{N}\qquad \mbox{such that both~~ $C_{I}x\leq b_{I}$ ~and ~$C_{E}x=b_{E}$}
$$
is equivalent to the condition
$$
\forall\, \lambda\in {\dv R}^{L}\qquad
\mbox{[\,$\lambda_{I}\geq \0_{I}$ ~and~ $C^{\top}\lambda=\0_{N}$\,]} ~~~\mbox{implies~~ $b^{\top}\lambda\geq 0$\,.}
$$
\end{lemma}
The reader can find this feasibility criterion in \cite[Theorem~9.2]{Chvatal1983LP}
in slightly modified formulation, namely as the equivalence of negations of the conditions from Lemma~\ref{lem.feasibility-criterion}.

\begin{lemma}\label{lem.exact-cone}\rm
Assume $|N|\geq 2$. For any set function $\m\in {\dv R}^{\caP}$, we have
\begin{enumerate}
\item $\m\in B(N)$ ~~iff~~ $\forall\, \theta\in \Theta^{N}_{N}\quad \llangle \theta,\m\rrangle \,\geq 0$\,,
\item $\m\in E(N)$ ~~iff~~ $\forall\:\emptyset\neq D\subseteq N ~~
\forall\, \theta\in \Theta^{N}_{D}\quad \llangle \theta,\m\rrangle\geq 0$\,.
\end{enumerate}
\end{lemma}

Because each of the cones $\Theta^{N}_{D}$ is a rational polyhedral cone by definition, the cones
$\Theta^{N}_{N}$ and $\mbox{cone}\, (\bigcup_{\emptyset\neq D\subseteq N} \Theta^{N}_{D})$ are generated by finitely
many rational vectors. This, together with Lemma~\ref{lem.exact-cone}, implies that both
$B(N)$ and $E(N)$ are rational polyhedral cones.

\begin{proof}
One can assume without loss of generality that $\m\in {\cal G}(N)$. Indeed,
otherwise $\m$ can be replaced by its shifted version $\mg\in {\cal G}(N)$ given by \eqref{eq.zero-projection}.
Since $\m=\mg+\m(\emptyset)\cdot\muo$ and any $\theta\in\Theta^{N}_{D}$, where $\emptyset\neq D\subseteq N$, satisfies $\sum_{T\subseteq N} \theta (T)=0$, it follows that
$$
\llangle \theta,\m\rrangle=\llangle \theta,\mg\rrangle + \m(\emptyset)\cdot
\underbrace{\llangle \theta,\muo\rrangle}_{=0} = \llangle \theta,\mg\rrangle =
\sum_{\emptyset\neq S\subseteq N} \theta(S)\cdot\mg(S).
$$
Realize that both equivalences $\m\in B(N)\Leftrightarrow\mg\in B(N)$ and
$\m\in E(N)\Leftrightarrow\mg\in E(N)$ hold.

Given a non-empty $D\subseteq N$ and a game $\m\in {\cal G}(N)$,
the existence of an element of the core $C(m)$ achieving the bound for $D$ can be characterized in terms of $\Theta^{N}_{D}$:
\begin{equation}
[\,\exists\, x\in C(\m)\quad \sum_{i\in D} x_{i}=\m(D)\,]
~~\Leftrightarrow~~
[\,\forall\, \theta\in \Theta^{N}_{D}\quad
\underbrace{\sum_{\emptyset\neq S\subseteq N} \theta(S)\cdot\m(S)}_{\llangle\theta ,\m\rrangle}\geq 0\,]\,.
\label{eq.tight-D}
\end{equation}
The choice $D=N$ in \eqref{eq.tight-D} proves the first equivalence in Lemma~\ref{lem.exact-cone};
the second one follows analogously by applying \eqref{eq.tight-D}
with all $\emptyset\neq D\subseteq N$.

To verify \eqref{eq.tight-D} we formulate the condition on its left-hand side of \eqref{eq.tight-D} as the feasibility condition
for a system of linear constrains from Lemma~\ref{lem.feasibility-criterion}.
Let $L$ be ${\cal P}(N)\setminus\{\emptyset\}$, with $E$ consisting of $D$ and $N$,
and $I$ being the rest of ${\cal P}(N)\setminus\{\emptyset\}$;
that is, $|E|=2$ if $D\neq N$ and $|E|=1$ if $D=N$. The matrix $C$ will have
the entries $-\chi_{S}(i)$ for $S\in {\cal P}(N)\setminus\{\emptyset\}$ and $i\in N$;
the component of the vector $b$ for $S\in {\cal P}(N)\setminus\{\emptyset\}$
will be $-\m(S)$.
Thus, by Lemma~\ref{lem.feasibility-criterion}, the condition is equivalent to the
requirement that, for each $\lambda\in {\dv R}^{{\cal P}(N)\setminus\{\emptyset\}}$ such that
\begin{enumerate}
\item $\lambda (S)\geq 0$ for each $S\subseteq N$ with the exception of sets $\emptyset,D,N$, [$\Leftrightarrow ~~\lambda_{I}\geq \0_{I}$]
\item $-\sum\limits_{T\subseteq N:\, i\in T} \lambda (T)=\sum\limits_{\emptyset\neq T\subseteq N} -\chi_{T}(i)\cdot \lambda(T)=0$ for any $i\in N$, [$\Leftrightarrow ~~C^{\top}\lambda=\0_{N}$]
\end{enumerate}
we get $\sum\limits_{\emptyset\neq T\subseteq N} -m(T)\cdot \lambda (T)\geq 0$ 
[$\Leftrightarrow ~~b^{\top}\lambda\geq 0$].

We put $\theta(S)=-\lambda(S)$ for $\emptyset\neq S\subseteq N$ and
$\theta(\emptyset)=\sum_{\emptyset\neq T\subseteq N} \lambda(T)$ and observe that
the itemized conditions in terms of $\theta$ mean that $\theta\in \Theta^{N}_{D}$
and the conclusion that $\llangle \theta,\m\rrangle\geq 0$.
Thus, the requirement is nothing but the condition on the right-hand side of \eqref{eq.tight-D}.
\qed \end{proof}

\subsection{Extreme rays of the dual cone to the balanced cone}\label{ssec.basic-balanced}
The following observation will be used later as an auxiliary fact.

\begin{lemma}\label{lem.balanced-decomp}\rm
If $|N|\geq 2$ then $\theta\in{\dv R}^{\caP}$ generates an extreme ray of
$\Theta^{N}_{N}$ iff it is a positive multiple of the coefficient vector
$\alpha_{\calB}$ for a non-trivial min-balanced system $\calB$ on $N$.
\end{lemma}

We would be able to give a direct proof of Lemma~\ref{lem.balanced-decomp},
analogous to our later proof of Lemma~\ref{lem.total-decomp}.
Nonetheless, we guess that the reader will prefer a shorter indirect proof based
on a well-known classic result reported earlier in Lemma~\ref{lem.balanced-cone}.

\begin{proof}
The first claim in Lemma~\ref{lem.exact-cone} says that $B(N)$ is the dual cone to $\Theta^{N}_{N}$.
Since $\Theta^{N}_{N}$ is a polyhedral cone, this implies that $B(N)$
and $\Theta^{N}_{N}$ are (polyhedral) cones which are dual each other (see Section~\ref{sec.basic-notions}). Therefore, the extreme rays of the pointed cone $\Theta^{N}_{N}$ (see Lemma~\ref{lem.dual-cones})
correspond to facets of $B(N)$. These are described by non-trivial min-balanced systems on
$N$ in Lemma~\ref{lem.balanced-cone}. In particular, $\theta$ is an extreme ray of $\Theta^{N}_{N}$
iff it has the form $\theta=k\cdot \alpha_{\calB}$ for some $k>0$ and a non-trivial min-balanced system
$\calB$ on $N$.
\qed \end{proof}

Of course, the observation from the above proof that the cones $\Theta^{N}_{N}$ and $B(N)$ are dual each other allows one to derive Lemma~\ref{lem.balanced-cone} as a consequence of the statement in Lemma~\ref{lem.balanced-decomp}.
Note that, in the original proof of Lemma~\ref{lem.balanced-cone} from \cite{Shapley1967On-balanced-set},
the dual cone to ${\cal B}(N)$ has been described in more complicated way, as the conic hull of an infinite set of coefficient vectors for balancing set systems.
Our simplification allows one to observe the following.

\begin{corollary}\label{cor.balan-inner}\rm
If $|N|\geq 2$ then the cone $B(N)$ is the conic hull of its linearity space $L(N)$ of modular set functions and the functions $-\delta_{S}$ for $\emptyset\neq S\subset N$.
\end{corollary}

Note that $L(N)$ is spanned by the functions $\pm\muo$ and $\pm\mui$ for $i\in N$. In particular, the extreme rays of the pointed cone
${\cal B}_{\ell}(N) \coloneqq \{ \m\in{\cal B}(N)\,:\ \m(\{i\})=0~~ \mbox{for $i\in N$}\}$
are those generated by functions $-\delta_{S}$ for $S\subseteq N$, $|S|\geq 2$.

\begin{proof}
Definition~\ref{def.Theta-cones} allows one to observe that
the linear hull of $\Theta^{N}_{N}$ is the set of \mbox{o-standardized} functions and
that (all) facets of $\Theta^{N}_{N}$ are defined by inequalities $\llangle\theta,-\delta_{S}\rrangle\geq 0$
for $\emptyset\neq S\subset N$. Thus, the dual cone $B(N)$ has the
orthogonal complement of the linear hull of $\Theta^{N}_{N}$ as its linearity space and the atomic faces of $B(N)$
correspond to facets of $\Theta^{N}_{N}$. These atomic faces correspond to the extreme rays of any pointed version of $B(N)$.
\qed \end{proof}

\subsection{Characterization of the totally balanced cone}\label{ssec.dual-total}

First we introduce a polytope generating the dual cone to $T(N)$.

\begin{definition}\rm\label{def.Delta-cones}
For any set $M\subseteq N$, $|M|\geq 2$, we denote by $\Delta_{M}$ the set of $\theta\in {\dv R}^{\caP}$ satisfying the conditions
\begin{enumerate}
    \item $\theta (S)\leq 0$, for $\emptyset\neq S\subset M$, $\theta(R)=0$ for $R\subseteq N$, $R\smin M\neq\emptyset$ and
    \item $\theta(\emptyset)=1$, $\sum_{S\subseteq N} \theta (S)=0, \sum_{T\subseteq N:\, i\in T} \theta (T)=0$ for any $i\in N$.
\end{enumerate}
Further, we introduce the polytope
$$
\Delta ~\coloneqq~ \mbox{conv}\, \left(\bigcup_{M\subseteq N, |M|\geq 2} \Delta_{M} \right).
$$
\end{definition}

Note that every $\Delta_{M}$ is a bounded polyhedron as
$$1=\theta(\emptyset)=\sum_{\emptyset\neq T\subseteq M\smin \{i\}} -\theta(T)\geq 0$$ for any  $i\in M$ and $\theta(M)=\sum_{S\subset M} -\theta(S)$. This makes the definition correct.

\begin{lemma}\label{lem.total-dual}\rm
If $|N|\geq 2$ and $\m\in {\dv R}^{\caP}$ then
$$
\m\in T(N) ~~\Leftrightarrow~~ \llangle\theta,\m\rrangle\geq 0\quad
\mbox{for any $\theta\in\Delta$.}
$$
\end{lemma}

\begin{proof}
If $\m\in T(N)$ and $\theta\in\Delta_{M}$ for some $M\subseteq N$, $|M|\geq 2$,
then the restriction of $\m$ to ${\cal P}(M)$ belongs
to $B(M)$, the restriction of $\theta$ to ${\cal P}(M)$ belongs to $\Theta^{M}_{M}$ and
$\llangle \theta, \m\rrangle=\sum_{S\subseteq M} \theta(S)\cdot \m(S)\geq 0$ by Lemma~\ref{lem.exact-cone}
applied to $N=M$. A convex combination of valid inequalities for $T(N)$ is a valid inequality
for $T(N)$ which gives $\llangle\theta,\m\rrangle\geq 0$ for $\theta\in\Delta$.

Conversely, if all the inequalities hold for $\m\in {\dv R}^{\caP}$
then, for any fixed $M\subseteq N$, $|M|\geq 2$, the restriction of $\m$ to ${\cal P}(M)$
satisfies $\sum_{S\subseteq M} \theta^{\prime}(S)\cdot \m(S)\geq 0$ for any $\theta^{\prime}\in\Theta^{M}_{M}$
with $\theta^{\prime}(\emptyset)=1$. By Lemma~\ref{lem.dual-cones}, any non-zero $\theta\in\Theta^{M}_{M}$ satisfies
$\theta(\emptyset)>0$ and is a positive multiple of such $\theta^{\prime}$. This implies the same inequalities for any $\theta\in\Theta^{M}_{M}$.
Hence, by Lemma~\ref{lem.exact-cone} applied to $N=M$, the restriction of $\m$ to ${\cal P}(M)$ belongs
to $B(M)$. The same is the case with $M\subseteq N$, $|M|=1$, which gives $\m\in T(N)$.
\qed \end{proof}

\subsection{Proof of Theorem~\ref{thm.total}}\label{ssec.basic-total}

\begin{lemma}\label{lem.total-decomp}\rm
A vector $\theta\in{\dv R}^{\caP}$ is an extreme point of
$\Delta$ iff it is $\alpha_{\calB}(\emptyset)^{-1}$-multiple of $\alpha_{\calB}$
for a non-trivial irreducible min-balanced system $\calB$ on some $M\subseteq N$, $|M|\geq 2$.
\end{lemma}

\begin{proof}
Let $\tilde{\alpha}_{\calB}$ denote the $\alpha_{\calB}(\emptyset)^{-1}$-multiple of $\alpha_{\calB}$ for any such system $\calB$. Every such vector $\tilde{\alpha}_{\calB}$ for a min-balanced system $\calB$
on $M\subseteq N$, $|M|\geq 2$, belongs to $\Delta_{M}$ (see the formula \eqref{eq.coef-def} in Section \ref{ssec.min-balan-ineq}) and, therefore,
it belongs to $\Delta$.

We first show that any extreme point $\theta$ of $\Delta$ has the form $\tilde{\alpha}_{\calB}$
for an irreducible min-balanced system $\calB$ on some $M\subseteq N$, $|M|\geq 2$. Because of the
definition of $\Delta$ there exists such $M$ that $\theta$ is an extreme point of $\Delta_{M}$.
In particular, the restriction of $\theta$ to ${\cal P}(M)$ generates an extreme ray of $\Theta^{M}_{M}$.
By Lemma~\ref{lem.balanced-decomp} applied to $N=M$ we derive that $\theta=\tilde{\alpha}_{\calB}$
for some non-trivial min-balanced system on $M$.
Corollary~\ref{cor.irreducible} then implies that $\calB$ must be
irreducible as otherwise $\theta$ is a convex combination of $\tilde{\alpha}_{\calC}\in\Delta$
for irreducible min-balanced systems ${\cal C}\subseteq {\cal P}(M)$.

The second step is to show that every $\tilde{\alpha}_{\calB}$
for a non-trivial irreducible min-balanced system $\calB$ on some $M\subseteq N$, $|M|\geq 2$,
is an extreme point of $\Delta$. As $\tilde{\alpha}_{\calB}\in\Delta$, by the former step, one can write
\begin{align*}
&\tilde{\alpha}_{\calB} = \sum_{i=1}^{t} \, \gamma_{i}\cdot \theta_{i}\qquad
\mbox{where $t\geq 1$, $\gamma_{i}>0$ for $i=1,\ldots, t$,~ $\sum_{i=1}^{t} \gamma_{i}=1$, ~\mbox{and}}\\
&  \forall\, i=1,\ldots, t,\; \theta_{i}=\tilde{\alpha}_{{\calB}_{i}}\; \mbox{for non-trivial min-balanced ${\calB}_{i}$ on $M_{i}\subseteq N$, $|M_{i}|\geq 2$.}
\end{align*}
We are going to show that $\theta_{i}=\tilde{\alpha}_{\calB}$ for all $i=1,\ldots, t$, which gives the extremity of $\tilde{\alpha}_{\calB}$. Realize that, for every $R\subseteq N$, one has $\muR\in T(N)$:
if $\emptyset\neq R\subseteq S$ for some $\emptyset\neq S\subseteq N$, then the restriction of $\muR$ to
${\cal P}(S)$ is in ${\cal B}(S)\subseteq B(S)$ because $\frac{1}{|R|}\cdot\chi_{R}\in {\dv R}^{S}$ is in its core.
Therefore, by Lemma~\ref{lem.total-dual}, $\llangle \theta_{i},\muR\rrangle\geq 0$
for any $i$ and $R$. For any $R\subseteq N$ with $R\smin M\neq\emptyset$,
$$
0=\sum_{S\subseteq N:\, R\subseteq S} \tilde{\alpha}_{\calB}(S)
=\llangle \tilde{\alpha}_{\calB},\muR\rrangle
=\sum_{i=1}^{t} \gamma_{i}\cdot \underbrace{\llangle \theta_{i},\muR\rrangle}_{\geq 0}
\quad\Rightarrow\quad \forall\, i\quad\llangle\theta_{i},\muR\rrangle=0\,.
$$
Hence, we observe by decreasing induction governed by cardinality of $R$ that,
for any $R\subseteq N$ with $R\smin M\neq\emptyset$
and $i=1,\ldots,t$, one has $\theta_{i}(R)=0$. Thus, for every $i$ one has $\calB_{i}\subseteq {\cal P}(M)$.

A crucial observation is that, in fact, $\calB_{i}\subseteq\calB$ for any $i=1,\ldots,t$. To verify that
assume for contradiction that there exists $j$ such that $\calB_{j}\setminus\calB\neq\emptyset$.
Consider inclusion minimal set $A\subseteq N$ such that $[\,\exists\, j\in\{1,\ldots,t\}\quad A\in\calB_{j}\setminus\calB\,]$.
By $\calB_{i}\subseteq {\cal P}(M)$ for all $i$ one has $A\subseteq M$.
The fact $A\in\calB_{j}$ implies, by Lemma~\ref{lem.min-balanced}, that
$\emptyset\neq A\subset \bigcup\calB_{j}=M_{j}\subseteq M$, which gives $A\subset M$.
Because $\emptyset\neq A\not\in\calB$ one has
$0\stackrel{\eqref{eq.coef-def}}{=}\tilde{\alpha}_{\calB}(A)=\sum_{i=1}^{t} \, \gamma_{i}\cdot \theta_{i}(A)$
(see Section \ref{ssec.min-balan-ineq}) and, since $j$ exists with $A\in\calB_{j}$ one has
$\theta_{j}(A)=\tilde{\alpha}_{{\calB}_{j}}(A)<0~\Rightarrow~ \gamma_{j}\cdot \theta_{j}(A)<0$
implying the existence of $k\in\{1,\ldots,t\}$ such that $\tilde{\alpha}_{{\calB}_{k}}(A)=\theta_{k}(A)>0$,
which means $A=M_{k}=\bigcup\calB_{k}$ (see Section \ref{ssec.min-balan-ineq}).
The inclusion minimality of $A$ means, for any $Z\subset A$ and $i$, that $Z\in\calB_{i} ~\Rightarrow~ Z\in\calB$.
In particular, any $Z\in\calB_{k}$ satisfies both $Z\subset A$ and $Z\in\calB$ and we have shown
$\calB_{k}\subseteq \calB_{A}=\{S\in\calB\,:\ S\subset A\}$.
Since $\calB_{k}$ is min-balanced on $M_{k}=A$ it implies that
$\chi_{A}$ is in the conic hull of $\{\chi_{S}\,:\, S\in\calB_{A}\}$. 

Now, for every $\varepsilon\geq 0$ we put
$$
\theta^{\varepsilon} ~\coloneqq~ \tilde{\alpha}_{\calB}+\varepsilon\cdot (\tilde{\alpha}_{\calB}-\theta_{k})
$$
and observe that, for sufficiently small $\varepsilon>0$, one has $\theta^{\varepsilon}\in\Delta_{M}$
(use Definition~\ref{def.Delta-cones} and realize that ${\cal B}$ is on $M$ and
$\calB_{k}\subseteq \calB_{A}\subseteq\calB$).
On the other hand, $\theta^{\varepsilon}\not\in\Delta_{M}$ for sufficiently large $\varepsilon>0$
(because $\Delta_{M}$ is bounded).
Take maximal $\varepsilon>0$ such that $\theta^{\varepsilon}\in\Delta_{M}$, which means that there exists
$B\subset A$ such that $\theta^{\varepsilon}(B)=0$ (such $B$ necessarily belongs to $\calB_{k}\subseteq\calB$).
The definition of $\Delta_{M}$ allows one to observe
that ${\cal C}\coloneqq\{ S\subseteq M\,:\, \theta^{\varepsilon}(S)<0\}$ is such that
$\chi_{M}$ is in the conic hull of $\{\chi_{T}\,:\, T\in{\cal C}\,\}$.
Indeed, realize that $\theta^{\varepsilon}(M)>0$
(because of $\bigcup\calB_{k}=A\subset M$) and that $\theta^{\varepsilon}$
vanishes outside ${\cal P}(M)$ (because $\theta^{\varepsilon}\in\Delta_{M}$).
The o-standardization condition from Definition~\ref{def.Delta-cones} implies, for every $i\in M$,
$$
\theta^{\varepsilon}(M) = 
- \sum_{T\subseteq M:\, i\in T} \theta^{\varepsilon}(T) =
- \sum_{T\subseteq M} \theta^{\varepsilon}(T)\cdot\chi_{T}(i)
$$
This gives
$$ \chi_{M}(i)=1=\sum_{T\subseteq M} \underbrace{\frac{-\theta^{\varepsilon}(T)}{\theta^{\varepsilon}(M)}}_{\geq 0}\cdot\chi_{T}(i)\,,
$$
which implies the claim about $\calC$.
Since ${\cal C}\subseteq \{A\}\cup(\calB\setminus \{B\})$ one has
$\chi_{M}$  is in the conic hull of $\{\chi_{T}\,:\, T\in\{A\}\cup(\calB\setminus \{B\})\,\}$.
This altogether means, by Definition~\ref{def.irreducible}, that $\calB$ is reducible, which
contradicts the assumption about $\calB$.

Thus, we are sure that $\calB_{i}\subseteq\calB$ for all $i=1,\ldots,t$.
If $\calB_{i}$ is such that $M_{i}=\bigcup\calB_{i}=M$ then
the assumption that $\calB$ is min-balanced on $M$ implies $\calB_{i}=\calB$ and
$\alpha_{{\cal B}_{i}}=\alpha_{\calB}$.
Hence, $\theta_{i}=\tilde{\alpha}_{{\calB}_{i}}=\tilde{\alpha}_{\calB}$ whenever $M_{i}=M$.
It remains to show that there is no $j\in\{1,\ldots,t\}$ such that $M_{j}\coloneqq\bigcup\calB_{j}\subset M$.
Assume for a contradiction that such $j$ exists and write
\begin{align*}
\tilde{\alpha}_{\calB} &= \sum_{i=1}^{t} \, \gamma_{i}\cdot \theta_{i}=
\sum_{i:\,M_{i}=M}  \gamma_{i}\cdot \theta_{i} + \sum_{j:\,M_{j}\subset M}  \gamma_{j}\cdot \theta_{j} \\ &
=\sum_{i:\,M_{i}=M}  \gamma_{i}\cdot \tilde{\alpha}_{\calB} + \sum_{j:\,M_{j}\subset M}  \gamma_{j}\cdot \theta_{j}.
\end{align*}
This implies
$$
(1-\sum_{i:\,M_{i}=M}  \gamma_{i})\cdot \tilde{\alpha}_{\calB}(M)
= \sum_{j:\,M_{j}\subset M}  \gamma_{j}\cdot \theta_{j}(M)= \sum_{j:\,M_{j}\subset M}  \gamma_{j}\cdot \tilde{\alpha}_{{\calB}_{j}}(M)=0,
$$
which gives a contradiction because both $1-\sum_{i:\,M_{i}=M}  \gamma_{i}>0$ and $\tilde{\alpha}_{\calB}(M)>0$.
Thus, one necessarily has $\theta_{i}=\tilde{\alpha}_{\calB}$ for all $i\in\{1,\ldots,t\}$
and the extremity of $\tilde{\alpha}_{\calB}$ is confirmed.
\qed \end{proof}

The following is a consequence of Lemma~\ref{lem.total-dual}.

\begin{corollary}\label{cor.total-decomp}\rm
If $|N|\geq 2$ and $\m\in {\dv R}^{\caP}$ then $\m\in T(N)$ iff
\begin{enumerate}
\item $\llangle\alpha_{\calB},\m\rrangle\geq 0$
for any non-trivial irreducible min-balanced system $\calB$ on $M\subseteq N$, where $|M|\geq 2$;
\end{enumerate}
which is equivalent to a formally stronger condition
\begin{enumerate}\setcounter{enumi}{1}
\item $\llangle\alpha_{\calB},\m\rrangle\geq 0$
for any non-trivial min-balanced system $\calB$ on $M\subseteq N$, where $|M|\geq 2$.
\end{enumerate}
\end{corollary}

\begin{proof}
The first statement follows directly from Lemma~\ref{lem.total-dual} and Lemma~\ref{lem.total-decomp}. The second condition is equivalent to the first one by Corollary~\ref{cor.irreducible}.
\qed \end{proof}

Theorem~\ref{thm.total} is now easy to prove.

\begin{proof}
By Lemma~\ref{lem.total-dual}, the cone $T(N)$ is dual to the cone $\tilde{\Delta}$
spanned by $\Delta$. The fact that $\tilde{\Delta}$ is a closed convex cone
then implies that $\tilde{\Delta}$ and $T(N)$ are each others dual cones (see Section~\ref{sec.basic-notions}).
Hence, the extreme points of $\Delta$, characterized in Lemma~\ref{lem.total-decomp}, correspond to facets of $T(N)$.
The coefficient vectors for inequalities are just the extreme points of $\Delta$, which are positive multiples of $\alpha_{\calB}$ for non-trivial irreducible min-balanced systems $\calB\subseteq\caP$.
\qed \end{proof}

\section{Conjecture concerning the exact cone}\label{sec.conjecture}
One of our research goals was to characterize facet-defining inequalities for the cone $E(N)$. Despite we have not succeeded to get an ultimate answer to that question we came to a sensible
conjecture about what are these inequalities. Because a crucial notion in our conjecture is the concept of an irreducible min-balanced system introduced in Section\,\ref{sec.irreducible}, we will formulate the conjecture in this paper.

\begin{table}[h]
\caption{Numbers of facets of $E(N)$ and of its types for $n=|N|\leq 5$.}\label{tab.facet-numbers}
\begin{tabular}{|l|c|c|c|c|} \hline
Number of players & $n=2$ & $n=3$& $n=4$& $n=5$ \\ \hline
Number of facets & $1$ & $6$& $44$& $280$ \\
Number of its permutational types & $1$ & $2$& $6$& $16$ \\ \hline
\end{tabular}
\end{table}

The first step towards the conjecture was computing all the facet-defining
inequalities for $E(N)$ in case $|N|\leq 5$; their numbers are shown in Table~\ref{tab.facet-numbers}.
If $|N|=2$, we have $E(N)=B(N)$. If $|N|=3$, then $E(N)$ coincides with the cone of supermodular functions.
The results in case that $|N|=4$ are given in Example~\ref{ex:E4} below.

\begin{example}\label{ex:E4}
We list all six permutational types of 44 facet-defining inequalities for $E(N)$ in case $|N|=4$. We present a type representative, a number of inequalities of this type, the induced set system (see Section~\ref{ssec.min-balan-ineq}), and indicate what is the conjugate inequality (see Definition~\ref{def.conjugate-ineq}).
\begin{enumerate}
\item $m(ab)-m(a)-m(b)+m (\emptyset)\geq 0$\hfill
$6\times$ \\
${\cal B}_{\alpha}=\{a,b\}$\hfill conjugate type 4.
\item $m(abc)-m(a)-m(bc)+m(\emptyset)\geq 0$\hfill
$12\times$\\
${\cal B}_{\alpha}=\{a,bc\}$\hfill conjugate type 5.
\item $2\cdot m(abc) -m(ab)-m(ac)-m(bc)+m(\emptyset)\geq 0$\hfill
$4\times$\\
${\cal B}_{\alpha}=\{ab,ac,bc\}$\hfill conjugate type 6.
\item $m(abcd)- m(acd)-m(bcd)+m(cd)\geq 0$\hfill
$6\times$\\
${\cal B}_{\alpha}=\{acd,bcd\}$\hfill  conjugate type 1.
\item $m(abcd) -m(ad)-m(bcd)+m(d)\geq 0$\hfill
$12\times$\\
${\cal B}_{\alpha}=\{ad,bcd\}$\hfill conjugate type 2.
\item $m(abcd)-m(ad)-m(bd)-m(cd)+2\cdot m(d)\geq 0$\hfill
$4\times$\\
${\cal B}_{\alpha}=\{ad,bd,cd\}$\hfill conjugate type 3.
\end{enumerate}
\end{example}

In case $|N|=5$ we have processed the results of computation
performed by Quaeghebeur \cite{Quaeghebeur09} in context of imprecise probabilities.
The point is that the concept of a coherent lower probability, used in that context,
corresponds to the notion of a normalized exact game. The reader is referred to \cite{MM17} for more details about the correspondence between some game-theoretical concepts and those appearing in the context of imprecise probabilities.

The next step was to classify the inequalities into their permutational types. Finally, we have analyzed the results from a theoretical point of view and formulated the following conjecture, which is known to be true in case $|N|\leq 5$.

\begin{conjecture}\label{conj.exact}
If $|N|\geq 3$ then the facet-defining inequalities for the cone $E(N)$ are just the inequalities corresponding to non-trivial irreducible min-balanced systems $\calB$ with $\bigcup\calB\subset N$ and their conjugate inequalities.
\end{conjecture}

The conjecture agrees with the fact that $E(N)$ is closed under reflection (see Lemma~\ref{lem.facet-conjugate}). As a consequence of our main result we obtain the following simpler version of the conjecture, which is formally weaker.

\begin{corollary}\label{cor.facet-extend}\rm
The validity of Conjecture~\ref{conj.exact} implies
$$
E(N)=T(N)\cap T^{*}(N)\qquad
\mbox{where~ $T^{*}(N) ~\coloneqq~ \{ \m\in {\dv R}^{\caP}\,:\ \mre\in T(N)\,\}$.}
$$
\end{corollary}
\noindent
Thus, in words, the weaker version of the conjecture is as follows (see Remark~\ref{rem.anti-dual}):
\begin{quote}
{\em A game $\m\in {\cal G}(N)$ is exact iff both $m$ and its anti-dual $-\mstar$ are totally balanced.}
\end{quote}

\begin{proof}
The inclusion $E(N)\subseteq T(N)$ and the fact that $E(N)$ is closed
under reflection (see Lemma~\ref{lem.reflection}) imply $E(N)\subseteq T^{*}(N)$.
Hence, the inclusion $E(N)\subseteq T(N)\cap T^{*}(N)$ surely holds.
To show the converse inclusion consider $\m\in T(N)\cap T^{*}(N)$ and, by Theorem~\ref{thm.total} and the formula \eqref{eq.conjug-relect},
observe that inequalities $\llangle \alpha_{\calB},\m\rrangle\geq 0$ and
$\llangle \alpha^{*}_{\calB},m\rrangle \stackrel{\eqref{eq.conjug-relect}}{=}\llangle \alpha_{\calB},\mre\rrangle\geq 0$
are valid for every non-trivial irreducible min-balanced system $\calB$ with $\bigcup\calB\subset N$.
This implies, by validity of Conjecture~\ref{conj.exact}, that $\m\in E(N)$.
Thus, $T(N)\cap T^{*}(N)\subseteq E(N)$.
\qed \end{proof}

In fact, we are able to prove the converse of the statement from Corollary~\ref{cor.facet-extend}, namely
that $E(N)=T(N)\cap T^{*}(N)$ implies the validity of the conjecture. Nonetheless,
because a detailed proof of that statement would require 10 additional pages of technicalities while such a result is only marginally relevant to the main topic of this paper,
we decided to omit that proof herein.

\section{Conclusions}
The prime focus of this paper is on the polyhedral cone of totally balanced games.
We have introduced a highly relevant concept of an {\em irreducible min-balanced\/} set system
(Definition~\ref{def.irreducible}).
Our main result, Theorem \ref{thm.total}, says that there is a one-to-one correspondence between facet-defining inequalities for the cone of totally balanced games and non-trivial irreducible min-balanced systems on subsets of the player set. We have only paid attention to the outer (= facial) description of the cone, whereas the problem of its inner description (= characterizing it as the conic hull of finitely
many vectors) seems to be relevant as well; this remains to be an open task.

Some of our minor results concern the cone of balanced games. The extended version of this cone
is closed under reflection transformation, which implies that every facet-defining inequality
for it is accompanied with a {\em conjugate\/} facet-defining inequality (Lemma~\ref{lem.facet-conjugate}).
We have re-visited a procedure that associates a facet-defining inequality with
a min-balanced set system and extended a well-known
classic result by Shapley \cite{Shapley1967On-balanced-set} saying that the facet-defining inequalities for the cone of balanced games correspond to non-trivial min-balanced systems on the whole player set $N$; see Section~\ref{ssec.min-balan-ineq}.
What we have shown is that a complementary set system to a min-balanced set system on $N$ is also
a min-balanced system on $N$ and gives a conjugate inequality (Corollary~\ref{cor.balanced-dual-N}).
Further side-result is the inner description of the cone of balanced games (Corollary~\ref{cor.balan-inner}).

Our tools made it also possible to contribute to the study of the cone of exact games.
The extended version of this cone is also closed under reflection transformation, which implies that
facet-defining inequalities for it come in pairs of mutually conjugate inequalities (Lemma~\ref{lem.facet-conjugate}).
In this paper we have formulated a conjecture about what are facet-defining inequalities for this
cone, which complies with the above observation (see Conjecture~\ref{conj.exact}).
Our hypothesis is supported by computations for a small number of players; thus, we
will concentrate on proving/disproving this conjecture in our future research.
Note in this context that the extremity of an exact game can be recognized by a relatively simple linear-algebraic test; see \cite[Proposition 4]{SK2018} for the details.

\appendix \label{app}

\section{Min-balanced systems for a small number of players}\label{sec.catalogue}
Here we give a list of all permutational types of non-trivial min-balanced systems for at most four players.
We present a type representative, indicate what is the type of the complementary system (Definition~\ref{def.complementary-syst}) and
whether the type is irreducible, give the standard inequality ascribed to the system (see Section~\ref{ssec.min-balan-ineq})
and say what is the number of systems of this type.
In order to shorten the notation we write $abc$ instead of $\{a,b,c\}$.

\subsection{Two players}

The only non-trivial min-balanced system on $N=\{a,b\}$ is as follows.

\begin{enumerate}
\item $\calB=\{a,b\}$ \hfill self-complementary, {\em irreducible}\\
$m(ab)-m(a)-m(b)+m(\emptyset)\geq 0$\hfill $1\times$
\end{enumerate}

\subsection{Three players}

The following are all three types of 5 non-trivial min-balanced systems on $N=\{a,b,c\}$.

\begin{enumerate}
\item $\calB=\{ a, b, c\}$\hfill complementary type 3.\\
$m(abc)-m(a)-m(b)-m(c)+2\cdot m(\emptyset)\geq 0$\hfill $1\times$
\item $\calB=\{ a, bc\}$ \hfill  self-complementary, {\em irreducible}\\
$m(abc)-m(a)-m(bc)+m(\emptyset)\geq 0$\hfill $3\times$
\item $\calB=\{ ab, ac , bc\}$\hfill complementary type 1., {\em irreducible}\\
$2\cdot m(abc)-m(ab)-m(ac)-m(bc)+m(\emptyset)\geq 0$~\hfill $1\times$
\end{enumerate}

Thus, one has two types of 4 irreducible min-balanced systems on $N=\{a,b,c\}$.

\subsection{Four players}

The following are all nine types of 41 non-trivial min-balanced system on $N=\{a,b,c,d\}$.

\begin{enumerate}
\item $\calB=\{ a, b, c, d\}$\hfill complementary type 9.\\
$m(abcd)-m(a)-m(b)-m(c)-m(d)+3\cdot m(\emptyset)\geq 0$\hfill
$1\times$
\item $\calB=\{a, b, cd\}$\hfill complementary type 6.\\
$m(abcd)-m(a)-m(b)-m(cd)+2\cdot m(\emptyset)\geq 0$\hfill
$6\times$
\item $\calB=\{ ab, cd\}$\hfill self-complementary, {\em irreducible}\\
$m(abcd)-m(ab)-m(cd)+m(\emptyset)\geq 0$\hfill $3\times$
\item $\calB=\{ a, bcd\}$\hfill self-complementary, {\em irreducible}\\
$m(abcd)-m(a)-m(bcd)+m(\emptyset)\geq 0$\hfill $4\times$
\item $\calB=\{a, bc, bd, cd\}$\hfill complementary type 8.\\
$2\cdot m(abcd)-2\cdot m(a)-m(bc)-m(bd)-m(cd)+3\cdot m(\emptyset)\geq 0$\hfill
$4\times$
\item $\calB=\{ ab, acd, bcd\}$\hfill complementary type 2., {\em irreducible}\\
$2\cdot m(abcd)-m(ab)-m(acd)-m(bcd)+m(\emptyset)\geq 0$\hfill
$6\times$
\item $\calB=\{ a, bd, cd, abc\}$\hfill self-complementary\\
$2\cdot m(abcd)-m(a)-m(bd)-m(cd)-m(abc)+2\cdot m(\emptyset)\geq 0$\hfill
$12\times$
\item $\calB=\{ ab, ac, ad, bcd\}$\hfill complementary type 5., {\em irreducible}\\
$3\cdot m(abcd)-m(ab)-m(ac)-m(ad)-2\cdot m(bcd)+2\cdot m(\emptyset)\geq 0$\hfill
$4\times$
\item $\calB=\{ abc, abd, acd, bcd\}$\hfill complementary type 1., {\em irreducible}\\
$3\cdot m(abcd)-m(abc)-m(abd)-m(acd)-m(bcd)+m(\emptyset)\geq 0$\hfill
$1\times$
\end{enumerate}

Thus, there are five types of 18 irreducible min-balances systems on $N=\{a,b,c,d\}$.

\end{document}